\documentclass[12pt]{amsart}
\usepackage{latexsym}
\usepackage{graphicx, color, epsfig, verbatim, subfigure, epsf} 
\usepackage{amssymb, amsmath,mathrsfs}
\usepackage[left=2.5cm,right=2.5cm,top=2.5cm, bottom=2.5cm ]{geometry}
\usepackage{graphicx}
\usepackage{tikz}

\usepackage{float}

\usepackage{setspace}
\usepackage[pagebackref,hypertexnames=false, colorlinks, citecolor=red, linkcolor=red]{hyperref}

\newcommand{\McC}{\raise.5ex\hbox{c}}

\newcommand{\K}{\mathcal{K}}

\newcommand{\D}{\mathbb{D}}
\newcommand{\T}{\mathbb{T}}

\newcommand{\C}{\mathbb{C}}

\newcommand{\R}{\mathbb{R}}
\newcommand{\p}{\mathfrak{p}}

\newtheorem{theorem}{Theorem}[section]

\newcommand{\CO}{\mathcal{K}}

\newtheorem{lemma}[theorem]{Lemma}
\newtheorem{conjecture}[theorem]{Conjecture}

\newtheorem*{theorem*}{Theorem}
\newtheorem*{conjecture*}{Conjecture}
\newtheorem{corollary}[theorem]{Corollary}
\newtheorem*{corollary*}{Corollary}
\newtheorem{proposition}[theorem]{Proposition}
\newtheorem*{proposition*}{Proposition}

\def\bb{\begin{color}{blue}}
\def\bg{\begin{color}{green}}
\def\br{\begin{color}{red}}
\def\eg{\end{color}}
\def\er{\end{color}}
\def\eb{\end{color}}

\theoremstyle{remark}
\newtheorem{remark}[theorem]{Remark}

\newtheorem{definition}[theorem]{Definition}

\author[Bickel]{Kelly Bickel$^\dagger$}
\address{Department of Mathematics, Bucknell University, 360 Olin Science Building, Lewisburg, PA 17837, USA.}
\email{kelly.bickel@bucknell.edu}
\thanks{$\dagger$ Research supported in part by National Science Foundation
DMS grant \#1448846.}

\author[Pascoe]{James Eldred Pascoe$^\ddagger$}
\address{Department of Mathematics, Washington University in St. Louis, 1 Brookings Drive, Campus Box 1146, St. Louis, MO 63130, USA.}
\email{pascoej@wustl.edu}
\thanks{$\ddagger$ Research supported by National Science Foundation Mathematical Science Postdoctoral Research Fellowship DMS 1606260.}

\author[Sola]{Alan Sola}
\address{Department of Mathematics, Stockholm University, Kr\"aftriket 6, 106 91 Stockholm, Sweden.}
\email{sola@math.su.se}

\begin{document}
\title[Level curves of RIFs]{Level curve portraits of rational inner functions}
\date{\today}


\begin{abstract}
We analyze the behavior of rational inner functions on the unit bidisk near singularities 
on the distinguished boundary $\mathbb{T}^2$ using level sets. 
We show that the unimodular level sets of a rational inner function $\phi$ can be parametrized with 
analytic curves and connect the behavior of these analytic curves to that of the zero set of $\phi$. We apply these results to 
obtain a detailed description of the fine numerical stability of $\phi$: for instance, we show that 
$\frac{\partial \phi}{\partial z_1}$ and $\frac{\partial \phi}{\partial z_2}$ always possess the 
same $L^{\mathfrak{p}}$-integrability on $\mathbb{T}^2$, and we obtain combinatorial relations between intersection multiplicities at singularities and vanishing orders for branches  of level sets.  We also present several new methods of constructing rational inner functions that allow us to prescribe properties of their zero sets, unimodular level sets, and singularities.

\end{abstract}
\maketitle
\tableofcontents

\section{Introduction and overview}\label{sec:intro}
 \subsection{Introduction}
A rational function of a complex variable can be described in terms of its zeros and poles, and the behavior of the function near these points is in principle easy to capture in terms of their integer orders. The exact location and nature of the zeros and poles of a one-variable rational function are decisive in many applications: for instance, critical points and poles determine much of the dynamical properties of a rational function in iteration theory, and the zeros and poles of rational functions in one variable govern the stability of associated systems in control theory. This latter fact, that the qualitative nature of a system is determined by the location of zeros of polynomials defining an associated rational function, leads to the important notion of a \emph{stable polynomial}, one that has all roots outside the unit disk (or the left half-plane, depending on context). 

When studying a rational function of several variables in a mathematical or engineering context, one is again led to consider points where numerator and denominator vanish, but now a new and subtle phenomenon manifests itself: simultaneous vanishing at a point does not necessarily lead to algebraic cancellation. Nevertheless, it may still happen that the rational function retains some smoothness and boundedness properties near a common zero of numerator and denominator, and this then leads to a rich geometric structure at this point.

This paper is devoted to a detailed study of singularities of a certain important class of rational functions in two variables. We work on the \emph{unit bidisk} 
\[\D^2=\{(z_1,z_2)\in \C^2\colon |z_1|<1, |z_2|<1\}\]
and are interested in zeros and singularities on the distinguished boundary of the bidisk, which we identify with the two-torus $\T^2=\T \times \T$, the Cartesian product of two copies of the unit circle $\T=\{z\in \C\colon |z|=1\}$. The distinguished boundary $\T^2$ supports the maximum modulus principle for the bidisk, and is determining for most of the function-theoretic questions we will address in this paper.
A \emph{rational inner function} (RIF) on the bidisk is a rational function $\phi\colon \D^2\to \C$ that is analytic and bounded in $\D^2$ and has $|\phi(\zeta)|=1$ for almost every $\zeta \in \T^2$. Examples of such functions are
\[-\frac{3z_1z_2-z_1-z_2}{3-z_1-z_2} \quad  \textrm{and}\quad -\frac{2z_1z_2-z_1-z_2}{2-z_1-z_2};\]
the first one is smooth on the closed bidisk $\overline{\D^2}$, but the second example exhibits what is known as a ``non-essential singularity of the second kind" at $(1,1)$: the function has a non-tangential limit at $(1,1)$ but the vanishing polynomials $2z_1z_2-z_1-z_2$ and $2-z_1-z_2$ do not share a common factor. 

The numerators and denominators in these examples can be obtained from each other by reflection in the unit circle. In fact, W. Rudin and E.L. Stout showed, \cite{RudStout65} and \cite[Chapter 5]{Rud69}, that all RIFs on the bidisk are of the form
\[\phi(z_1,z_2)=e^{i\alpha}z_1^Mz_2^N\frac{\tilde{p}(z_1,z_2)}{p(z_1,z_2)}\]
where $\alpha$ is a real number, $M$ and $N$ are non-negative integers, $p$ is a \emph{semi-stable polynomial},  and the polynomial
\[\tilde{p}(z_1,z_2)=z_1^mz_2^n\overline{p\left(\frac{1}{\overline{z}_1},\frac{1}{\overline{z}_2}\right)}\]
is the \emph{reflection} of $p$. The pair of integers $(m,n) \in \mathbb{N}^2$ is referred to as the \emph{bidegree} of $p$ and is given by the largest powers of $z_1$ and $z_2$ that appear in $p$. A polynomial $p\in \C[z_1,z_2]$ is said to be semi-stable if it has no zeros in $\D^2$; it is \emph{stable} (or strictly stable) if it is non-vanishing on the closed bidisk. For simplicity, we usually consider rational inner functions of the form $\phi=\frac{\tilde{p}}{p}$ in this paper; monomial factors do not materially affect our conclusions.

The study of rational inner functions and semi-stable polynomials has a rich tradition in complex analysis \cite{KV79,McD87, AMY12, Kne10}, operator theory \cite{AglMcC, bk13, BLPreprint, Grinshpetal17,  bg17}, algebraic geometry \cite{AMS06, AMS08}, and systems theory and engineering \cite{Kum02, BSV05, GW06, Wag11}. We refer the reader to references provided in these papers for further work on these topics.
Recently, Knese \cite{Kne15} initiated the study of $L^2(\T^2)$-integrability of rational functions of the form $q/p$, where $p$ is assumed semi-stable but not necessarily strictly stable. In \cite{BPS17}, the authors derived a concrete relationship between the numerical stability of a rational inner function $\phi$, as measured by the $L^\p(\T^2)$-integrability of $\frac{\partial \phi}{\partial z_1}$ and $\frac{\partial \phi}{\partial z_2}$, and ``fine semi-stability" of its zero set, captured by \emph{contact orders} at a singularity. These measure how fast the zero set of $\tilde{p}$ approaches $\T^2$ in relation to how the fast the zero set approaches the singularity,  if one variable is restricted to $\T$. Informally, contact order can be defined for $\phi=\frac{\tilde{p}}{p}$ as follows. Setting
\[\mathcal{Z}_{\tilde{p}}=\{z\in \C^2\colon \tilde{p}(z)=0\},\]
we define the facial varieties
\[\mathcal{Z}_{\tilde{p}}^1=\mathcal{Z}_{\tilde{p}}\cap \left(\overline{\D}\times \T\right)  \quad \textrm{and}\quad \mathcal{Z}_{\tilde{p}}^2=\mathcal{Z}_{\tilde{p}}\cap \left(\T\times \overline{\D}\right).\]
The $z_i$-{\it contact order of $\phi$} is given by the largest number $K_i$ such that there exists a
sequence $\{w_k\} \subseteq \mathcal{Z}_{\tilde{p}}^i$ converging  to a singular point $\tau \in \T^2$ of $\phi$ and a positive constant $C$ such that
\[ \mathrm{dist}\big (w_k, \T^2 \big )\leq C \ \mathrm{dist} \big(w_k, \tau \big)^{K_i} \qquad \forall k \in \mathbb{N}.\]
(The precise definition of contact order is given in \cite[Section 2]{BPS17} and Section \ref{sec:prelim} below.)

In this paper, we study the numerical stability of an RIF and the geometry of its zero set via level curves of the RIF restricted to the two-torus. This approach allows us to ``visualize" the geometry of singularities of an RIF on $\T^2$ in a concrete and appealing way. More precisely, one of our main goals is to show how to divine ``fine semi-stability", that is, compute contact orders and related quantities by examining 
unimodular level curves 
\[\mathcal{C}_{\lambda}=\{\zeta \in \T^2\colon \phi(\zeta)=\lambda\}, \quad \textrm{for}\quad \lambda \in \T,
\]
and how they come together at singularities of $\phi$ on the two-torus. We show that such level curves are in fact smooth, in the sense that they can be parametrized by analytic functions. From this fact we are able to derive many properties of $\phi$ at its singularities, including for instance that its first partials enjoy the same $L^\p(\T^2)$-integrability properties.
Using smoothness of level curves together with certain embedding constructions, we are further able to apply our results concerning rational inner functions to draw conclusions about how special varieties in $\C^2$ intersect the two-torus. 

\subsection{Overview}
We proceed with an overview of the results contained in this paper: in what follows, these results are stated in a non-technical way, with references to precise versions in the body of the paper.
Several of our theorems, which are valid for general rational inner functions, can be illustrated by examining the simple rational inner function
\begin{equation}
\phi(z_1,z_2)=-\frac{2z_1z_2-z_1-z_2}{2-z_1-z_2}.
\label{faveform}
\end{equation}

\begin{figure}[h!]
    \subfigure[A family of level curves (black), with value curve (red).]
      {\includegraphics[width=0.44 \textwidth]{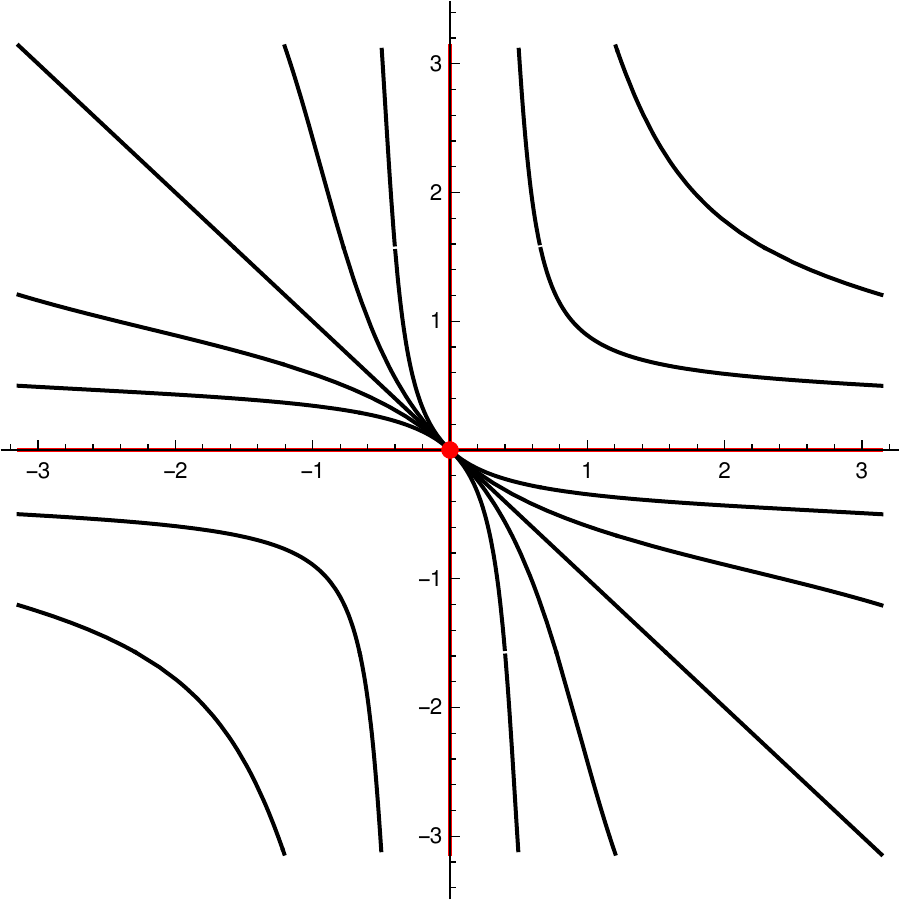}}
    \hfill
    \subfigure[Level curve $\mathcal{C}^*_1{(1,1)}$, corresponding to the nontangential value $\phi(1,1)=1$.]
      {\includegraphics[width=0.44 \textwidth]{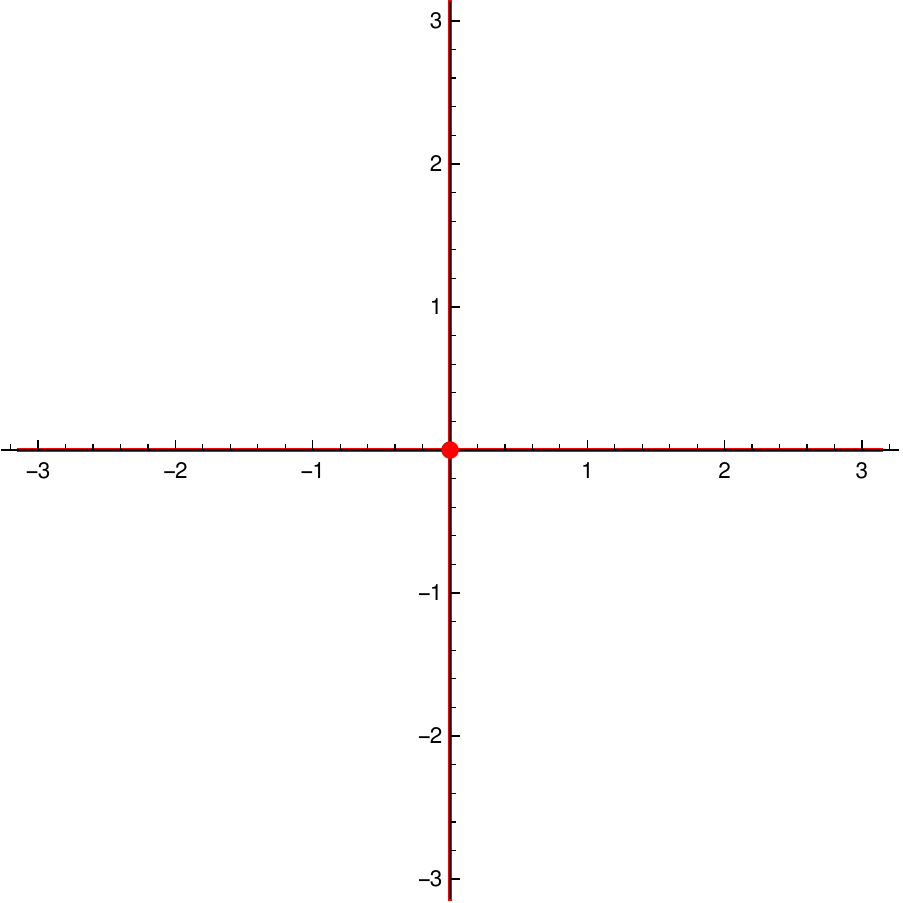}}
  \caption{\textsl{Level curves for $\phi(z_1,z_2)=-(2z_1z_2-z_1-z_2)/(2-z_1-z_2)$}.}
  \label{faveplots}
\end{figure}

The RIF $\phi$ has a single singularity at $(1,1)$ and a computation shows that $\phi(1,1)=1$, in the sense of non-tangential limits. By solving $-(2z_1z_2-z_1-z_2)=2-z_1-z_2$ we find that the level curve of $\phi$ corresponding to this value coincides with the union of coordinate axes, 
\[\mathcal{C}^*_{1}(1,1)=\{(e^{it_1}, 1)\}\cup\{(1,e^{it_2})\},\]
and thus consists of smooth components with a transversal intersection. For $\lambda\in \T\setminus\{1\}$, the associated level curves of $\phi$ are smooth and are described by 
\[\mathcal{C}_{\lambda}=\{(z_1,\psi^{\lambda}(z_1))\colon z_1\in \T\}\quad \quad \textrm{where}\quad\psi^{\lambda}(z_1)=\lambda\frac{1-\frac{1-\overline{\lambda}}{2}z_1}{z_1-\frac{1-\lambda}{2}},\]
the reciprocal of a M\"obius transformation of the disk. A plot of level curves, including the value curve, is provided in Figure \ref{faveplots}. Here, and throughout, we identify the two-torus with $(-\pi,\pi]\times (-\pi, \pi]$ for computational purposes. Thus the point $(1,1)\in \T^2$ corresponds to $(0,0)$ in our plots.

We now observe that all level curves pass through the singularity at $(1,1)$ in the second and fourth quadrants, and any pair of level curves with the exception of $\mathcal{C}^*_1(1,1)$ touch to order $2$ at the origin: that is, for any pair $\lambda, \mu \in \T\setminus \{1\}$,
\[|\psi^{\lambda}(z_1)-\psi^{\mu}(z_1)|\asymp |1-z_1|^2 \quad \textrm{as} \quad z_1 \to 1.\]
 The first fact illustrates what was called a Horn Lemma in \cite{BPS17}: level curves of an RIF are highly constrained in the way they pass through singularities. A precise formulation is given in Lemma \ref{lem:horn}. We use the Horn Lemma to prove one of the main results of our paper, namely that smoothness of unimodular level curves holds for any RIF. 
\begin{theorem*}[\ref{thm:smooth}]
The components of each unimodular level curve $\mathcal{C}_{\lambda}$ of a rational inner function $\phi$ can be parametrized by analytic functions.
\end{theorem*} 
The fact that we may have to resolve a level curve into components is illustrated above by the splitting of $\mathcal{C}^*_{1}(1,1)$ into horizontal and vertical axes. 
 
ln \cite{BPS17}, the contact orders of the rational inner function \eqref{faveform} with respect to both variables were computed, and were both found to be equal to $2$. In this paper, we show that contact order for a general RIF can be computed from unimodular level curves. 
\begin{theorem*}[\ref{thm:CO}]
The $z_1$-contact order of $\phi$ at a singularity $\tau \in \T^2$ is determined by the maximal order of vanishing of
$\psi^{\lambda}_i-\psi^{\mu}_j$ at $\tau$, where $z_1=\psi^{\mu}_i(z_2)$ ($i=1, \ldots, m$) and $z_1=\psi^{\nu}_j(z_2)$ ($j=1, \ldots, n$) are parametrizations of the branches of unimodular level curves of $\phi$ for two generic values of $\mu,\nu\in \T$.
\end{theorem*}
The precise meaning of ``generic" in this context will be discussed later in the paper.

The rational inner function \eqref{faveform} is symmetric in $z_1$ and $z_2$, and hence its $z_1$- and $z_2$-contact orders have to be equal. 
Using the fact that contact order is witnessed by unimodular level curves, we are able to prove that this, perhaps somewhat surprisingly, is true for any RIF.
\begin{theorem*}[\ref{thm:ECO}]
The $z_1$- and $z_2$-contact orders of a rational inner function are equal at each singularity.
\end{theorem*}
This means that we can speak of {\it the} contact order $\K_{\tau}$ of an RIF at a singularity $\tau \in \T^2$. The \emph{global contact order} $K$ of $\phi$ is the maximum of $\K_{\tau}$ over all singularities $\tau \in  \T^2$ of $\phi$. This, together with work in \cite{BPS17}, then implies that the first partials $\frac{\partial \phi}{\partial z_1}$ and $\frac{\partial \phi}{\partial z_2}$ of a rational inner function have the same $L^\p$-integrability properties.
\begin{theorem*}[\ref{thm:integral} and \ref{cor:integral}] For a rational inner function $\phi$ and for $1\leq \p< \infty$, we have
\[\frac{\partial \phi}{\partial z_1} \in L^\p(\T^2)\iff K<\frac{1}{\p-1}\iff  \frac{\partial \phi}{\partial z_2}\in L^\p(\T^2).\]
\end{theorem*}
In fact, we also establish a local $L^\p$-integrability version of this result.
\subsection{Refined results for complicated singularities}
The full strength of some of our results are best illustrated by considering more complicated examples of RIFs. In fact, a secondary objective of our work is to provide examples of RIFs $\phi=\tilde{p}/p$ that allow for detailed analysis while going beyond the $\deg p=(n,1)$ case, which is frequently easier to handle \cite{BLPreprint, Pas, BPS17}. 

Consider the bidegree $(2,1)$ rational inner function
\begin{equation}
\phi(z_1,z_2)=-\frac{4z_1^2z_2-z_1^2-3z_1z_2-z_1+z_2}{4-3z_1-z_2-z_1z_2+z_1^2}
\label{AMYfunction}
\end{equation}
which appears in \cite{AMY12} as an example of a function having a $C$-point at its singularity at $(1,1)$; this entails $\phi$ having higher-order non-tangential regularity. We have $\phi(1,1)=1$, and in \cite[Section 4]{BPS17}, 
it was shown that $\phi$ has contact orders equal to $4$ at its singularity.

\begin{figure}[h!]
    \subfigure[A family of level curves (black), with value curve (red).]
      {\includegraphics[width=0.4 \textwidth]{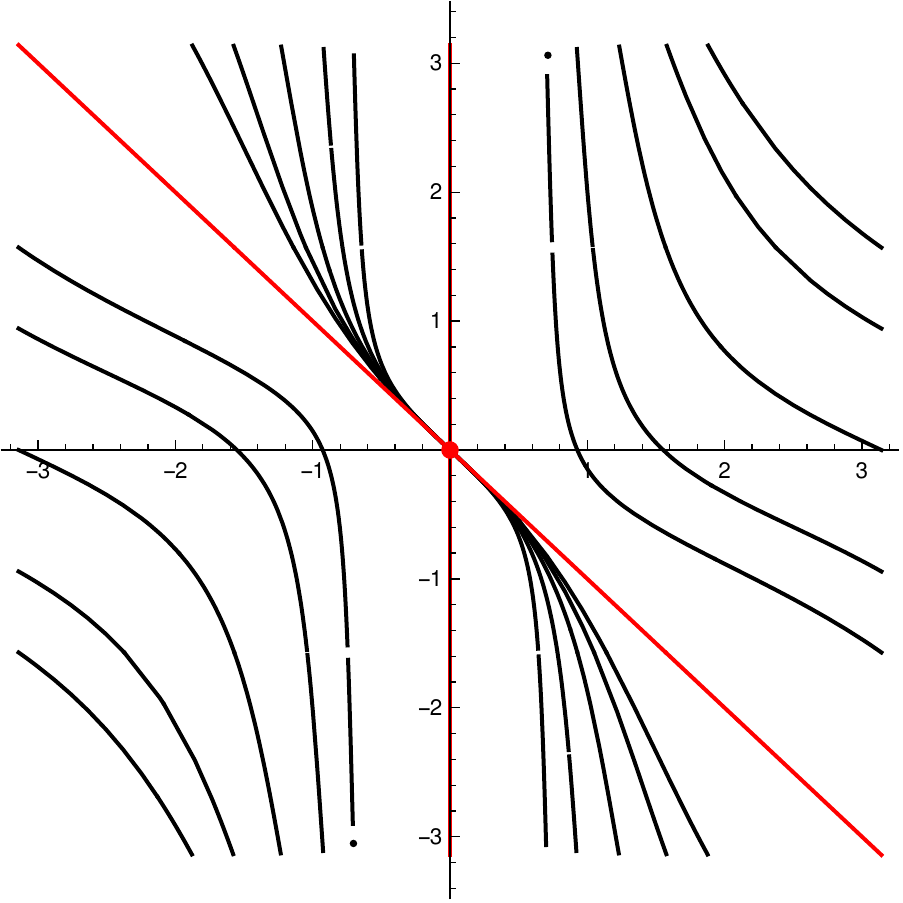}}
    \hfill
    \subfigure[Value curve consisting of vertical axis and anti-diagonal.]
      {\includegraphics[width=0.4 \textwidth]{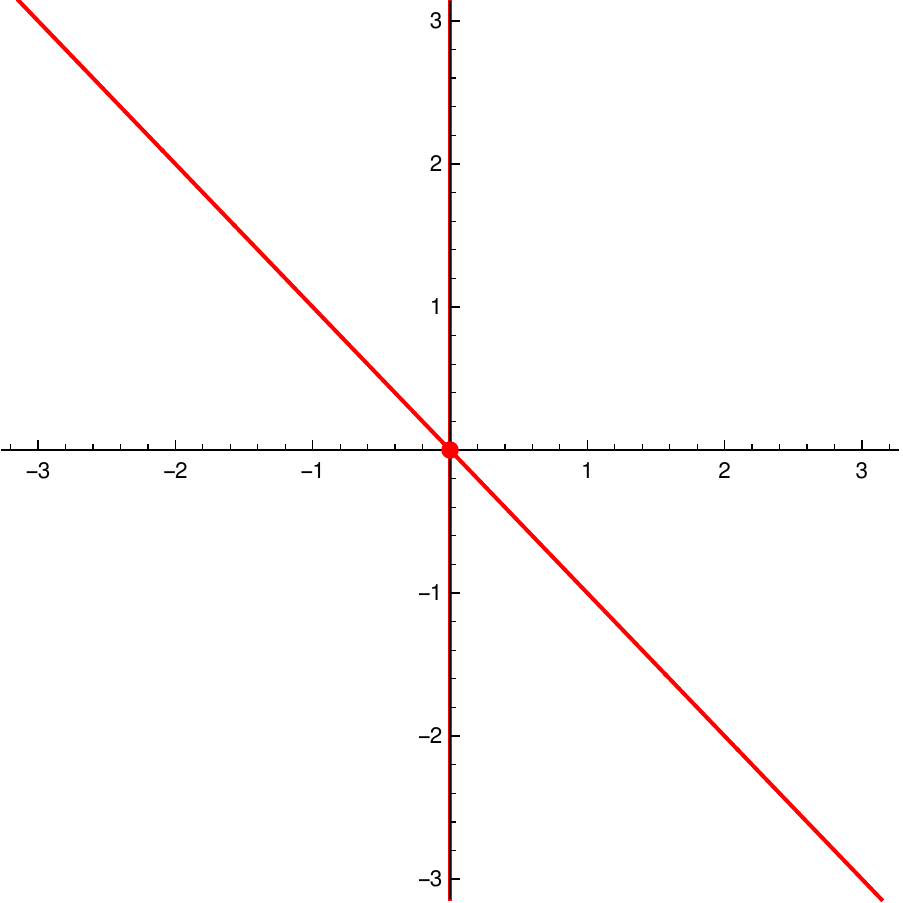}}
  \caption{\textsl{Level curves for the RIF in \eqref{AMYfunction}, exhibiting contact order equal to $4$}.}
  \label{AMYplots}
\end{figure}

These facts can again be seen by examining level sets. Setting $\tilde{p}=p$ yields the equation
\[4(z_1z_2-1)(z_1-1)=0\]
and thus, the level curve associated with the non-tangential value, which we call the value curve, is 
\[\mathcal{C}_1^*(1,1)=\{(1,e^{it_2}))\}\cup \{(e^{it_1},e^{-it_1})\}\]
again a union of smooth curves. By solving $\tilde{p}=\lambda p$ for $\lambda \in \T$ for $z_1$, we obtain a parametrization of level lines by
\[z_2=\psi^{\lambda}(z_1)=\frac{4\lambda-3\lambda z_1+\lambda z_1^2+z_1^2+z_1}{4z_1^2+\lambda z_1-3z_1+\lambda+1}, \quad z_1\in \T.\]

These smooth level curves are shown in Figure \ref{AMYplots} and one can again check by hand that generic level curves meet to order $4$, as guaranteed by Theorem \ref{thm:CO}.
Note that the slanted cross also appears as the value curve for the rational inner function 
\[\varphi(z_1,z_2)=-\frac{2z_1^2z_2-z_1-1}{2-z_1z_2-z_1^2z_2},\]
which was studied in \cite[Section 12]{BPS17}. There, it was computed that this $\varphi$ has contact order $\CO_{(1,1)}=2$ and
hence, a level curve alone does not determine contact order of an RIF: we need at least two level curves. In fact, Theorem \ref{thm:CO} allows for one omitted value $\mu_0 \in \T$, and we call the level curve corresponding to this value the \emph{exceptional level curve}.  As we have seen in our examples, the value curve associated to the non-tangential value at a singularity, exhibits some special features: frequently, the value curve coincides with the exceptional curve, but this is not always the case, as we show by example in Section \ref{sec:zoo}. Level curves that are neither value curves nor exceptional curves will be called \emph{generic}.

The two examples we have discussed so far have the special property that there is only one branch of $\mathcal{Z}_{\tilde{p}}$ coming in to the singularity. In general, however,
several branches of the zero set may come together, and these branches may individually exhibit different contact orders. Similarly, level curves may consist of several components.
In Section \ref{sec:COfine}, we analyze relations between branches of the zero set $\mathcal{Z}_{\tilde{p}}$ and branches of unimodular level curves. 
\begin{theorem*}[\ref{prop:bijection1}]
For a generic $\lambda \in  \T$, suppose $\mathcal{C}_{\lambda}$ is parametrized by finitely many functions $z_1=\psi^{\lambda}_1, \ldots, z_1=\psi^{\lambda}_L$ and $\mathcal{Z}_{\tilde{p}}$ has $L_0$ branches coming into a singularity on $\T^2$. Then $L\geq L_0$. Given two generic $\lambda, \mu \in \T$, and possibly after reordering, the contact order of a branch of $\mathcal{Z}_{\tilde{p}}$ is at most the order of contact between two matching level curves $z_1=\psi_i^{\lambda}(z_2)$ and $z_1=\psi_i^{\mu}(z_2)$.
\end{theorem*}
We conjecture that the converse statement is also true, so that we have a genuine bijection between contact order of individual branches of $\mathcal{Z}_{\tilde{p}}$ and components of level curves.

To illustrate this bijection,  we consider the bidegree $(4,2)$ polynomial
\begin{multline}
p(z_1,z_2)=-32 + 38 z_2 - 10 z_2^2 + 34 z_1 - 32 z_1z_2 + 2 z_1z_2^2 - 30 z_1^2 + 36 z_1^2 z_2\\- 
 2 z_1^2z_2^2 + 10 z_1^3 - 8 z_1^3z_2 - 6 z_1^3z_2^2 - 6 z_1^4 + 14 z_1^4z_2 - 
 8z_1^4z_2^2
 \label{MBMp}
\end{multline}
and its reflection
\begin{multline}
\tilde{p}(z_1,z_2)=-8 + 14 z_2 - 6 z_2^2 - 6 z_1 - 8 z_1z_2 + 10 z_1z_2^2 - 2 z_1^2 + 36 z_1^2z_2 - 
 30 z_1^2 z_2^2 \\+ 2 z_1^3 - 32 z_1^3z_2 + 34 z_1^3z_2^2 - 10 z_1^4 + 38 z_1^4z_2 - 
 32 z_1^4z_2^2
 \label{MBMptilde}
\end{multline}
and set $\phi=\tilde{p}/p$. This example can be obtained using a construction devised by the second author in \cite{Pas}; we provide a more detailed overview of this method in Section \ref{sec:zoo}.

The rational inner function $\phi$ has two singularities, at $(1,1)$ and $(-1,1)$ respectively. Taking radial limits reveals that  
$\phi(1,1)=1$ and $\phi(-1,1)=-1$. A computation using computer algebra shows that the associated intersection multiplicities (see Section \ref{sec:prelim} for a definition) are $N_{(1,1)}(p,\tilde{p})=14$ and $N_{(-1,1)}(p,\tilde{p})=2$, so that
\[16=N(p,\tilde{p})=N_{\T^2}(p,\tilde{p})=14+2,\]
and hence $p$ and $\tilde{p}$ have no further common zeros in $\mathbb{C}_{\infty}\times \mathbb{C}_{\infty}$ by B\'ezout's theorem. 

\begin{figure}[h!]
    \subfigure[Moduli of roots of $\tilde{p}(z)=0$ as functions of $z_2=e^{it_2}\in \T$.]
      {\includegraphics[width=0.4 \textwidth]{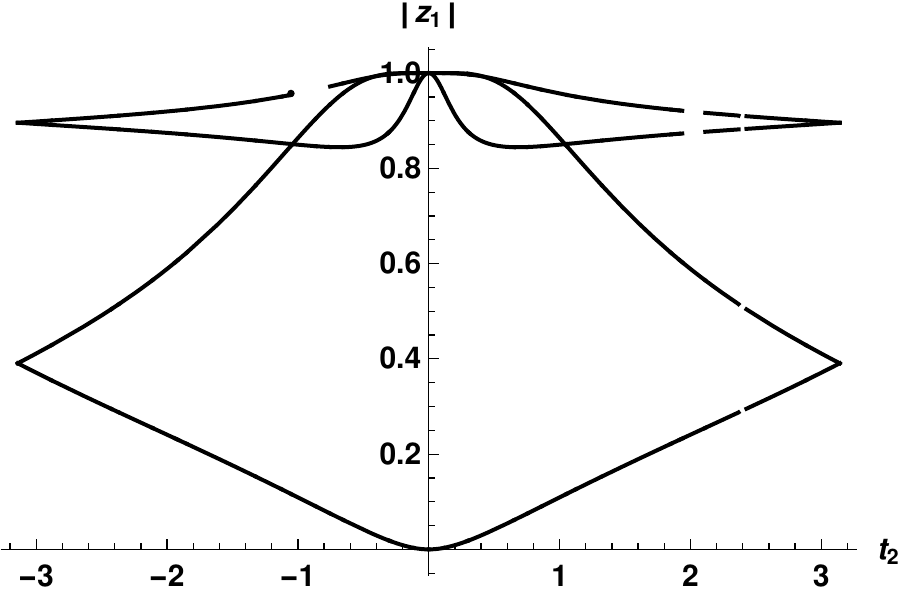}}
    \hfill
    \subfigure[Moduli of roots of $\tilde{p}(z)=0$ as functions of $z_1=e^{it_1}\in \T$.]
      {\includegraphics[width=0.4 \textwidth]{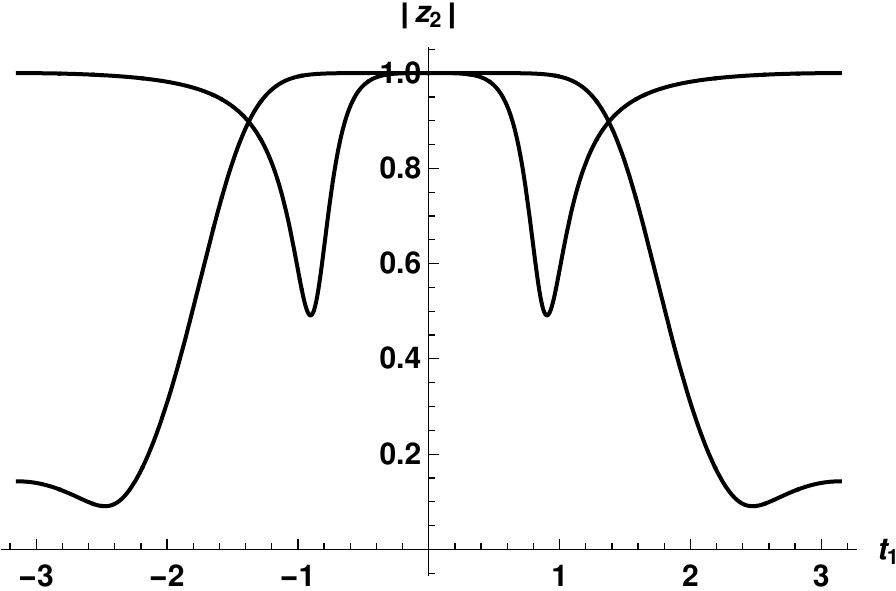}}
  \caption{\textsl{Solutions to $\tilde{p}(z_1,z_2)=0$ on the unit circle for $\tilde{p}$ in \eqref{MBMptilde} }.}
  \label{multiBMplotsroots}
\end{figure}
At the level of zero sets, a single branch of $\mathcal{Z}_{\tilde{p}}$ comes in to $(-1,1)$ with contact order $2$. At $(1,1)$, on the other hand, two branches of $\mathcal{Z}_{\tilde{p}}$ meet: one branch makes contact with the torus to order $4$, while the other has contact order $8$.  This can be seen by solving $\tilde{p}(z_1,z_2)=0$ for $z_1$ and $z_2$, respectively, and displaying the moduli of the resulting roots as functions on the unit circle: the rate at which these quantities approach $1$ is how contact order was originally defined in \cite{BPS17}. There are four branches on the left in Figure \ref{multiBMplotsroots}: one of these does not meet the torus. One of them has $1\mapsto -1$ and corresponds to the point $(-1,1)$ where contact order is $2$. The remaining two functions correspond to the branches meeting at $(1,1)$, one reaching $1$ with order $4$ and the other one with order $8$. On the right, there are two branches: one function takes on modulus $1$ once only, to order $8$, and the other takes on modulus $1$ twice, with order $4$ and $2$ respectively. Since global contact order is defined as a maximum over branches, we have overall contact order $8$ at $(1,1)$. 

The same arrangement is visible in Figure \ref{multiBMplots}, illustrating a bijection that exists between branches of zero sets and level curves of $\phi$,  now consisting of multiple components. Level curves trapped in the left-most horn at $(1,1)$ have order of contact equal to $4$, while level curves contained in the horn bounded by the vertical axis have order of contact $8$. We thus again obtain global contact order by maximizing over orders of contact. 

\begin{figure}[h!]
    \subfigure[Level curves.]
      {\includegraphics[width=0.4 \textwidth]{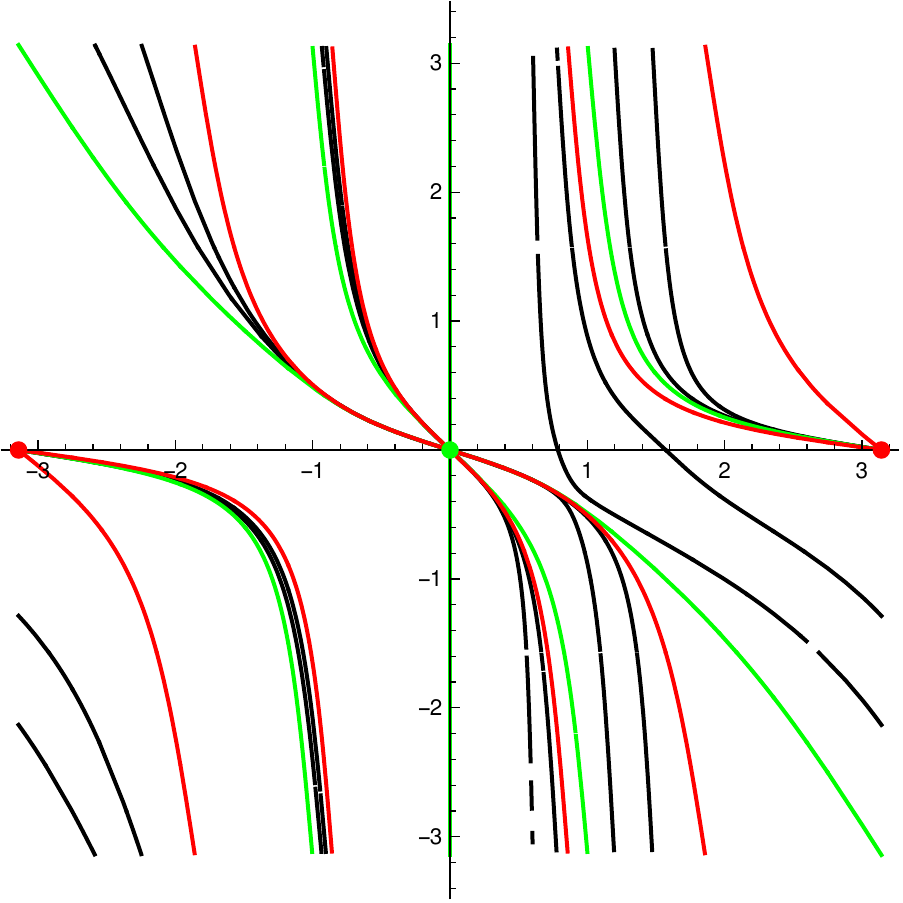}}
    \hfill
    \subfigure[Value curves $\mathcal{C}^*_1(1,1)$ (green) and $\mathcal{C}^*_{-1}(-1,1)$ (red).]
      {\includegraphics[width=0.4 \textwidth]{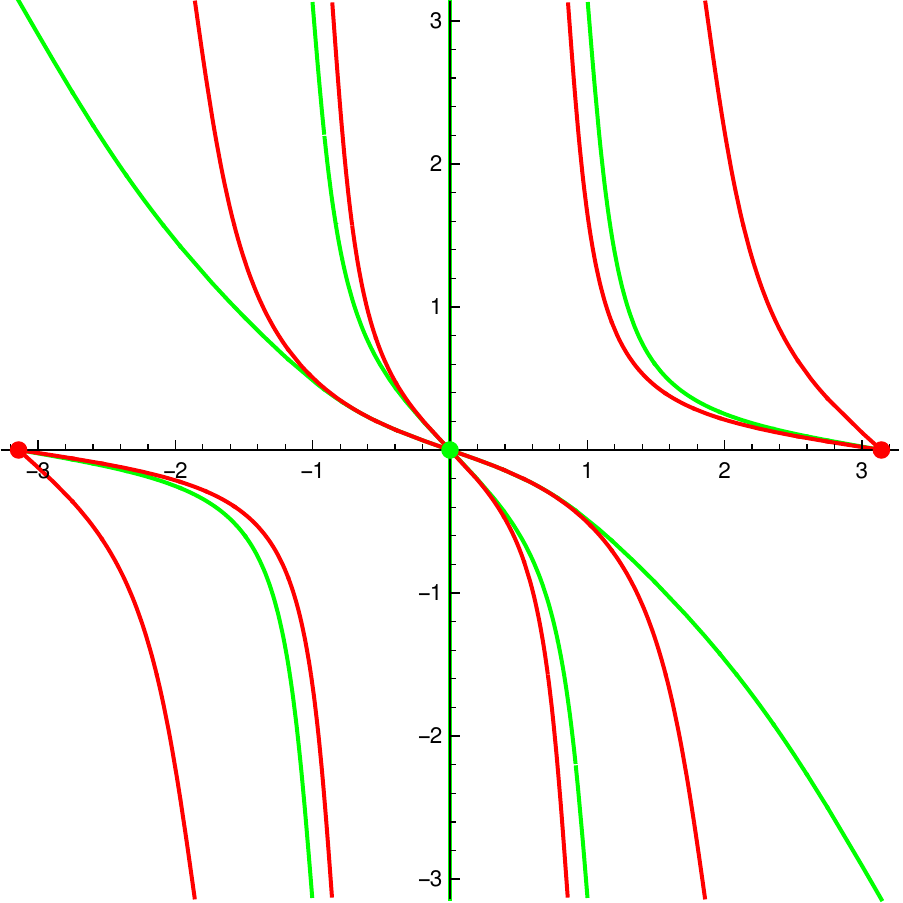}}
  \caption{\textsl{Level curves for $\phi=\tilde{p}/p$ constructed from \eqref{MBMp} and \eqref{MBMptilde}, an RIF with two singularities and multiple zero set branches}.}
  \label{multiBMplots}
\end{figure}

As is to be expected, intersection multiplicity and contact order at a singularity are related, even if they are in general different, as the example above shows. For instance, we prove the following result.
\begin{proposition*}[\ref{prop:COvsIM}]
The intersection multiplicity of $\mathcal{Z}_p$ and $\mathcal{Z}_{\tilde{p}}$ at a singularity $\tau \in \T^2$ of $\phi=\tilde{p}/p$ is bounded by the sum over pairwise minima of contact orders of branches of $\mathcal{Z}_{\tilde{p}}$ coming together at $\tau$.
\end{proposition*}

In terms of applications, our results have ramifications for codistinguished varieties. These varieties meet the closed bidisk along an infinite set in $\T^2$ and they arise as zero sets of polynomials $r$ with $r=\lambda \tilde{r}$ for a constant $\lambda \in \T$; Knese calls such polynomials {\it essentially $\T^2$-symmetric} \cite{Kne10TAMS}. Codistinguished varieties and their distinguished relatives appear in connection with Riemann surfaces \cite{RudTAMS69}, multivariable operator theory and determinantal representations \cite{AM05, Kne10TAMS, PS14}, interpolation \cite{JKS12}, as well as cyclicity problems for shift operators \cite{BKKLSS15}. Note that the value curves of the examples above can be seen to arise as $\mathcal{Z}_r\cap \T^2$ for a codistinguished variety $\mathcal{Z}_r$. We observe in Lemma \ref{thm:LS} (as has Knese \cite{Kne10TAMS}) that any curve in $\T^2$ of this form can be embedded as a level curve of an RIF, and since all such curves are smooth, we then obtain 
\begin{corollary*}[\ref{cor:codist}]
For any codistinguished variety $\mathcal{Z}_{r}$, the set $\mathcal{Z}_r\cap \T^2$ consists of smooth components.
 \end{corollary*}
In the same section we also present a characterization of when two zero sets $\mathcal{Z}_r$ and $\mathcal{Z}_q$ can be embedded as two different level curves of the same RIF.

\subsection{Structure of the paper}
We begin Section \ref{sec:prelim} by stating some preliminary results and collecting background material including Puiseux series expansions, intersection multiplicities, the definition of contact order, and the Horn Lemma, which describes approach regions for unimodular level curves of an RIF near singularities. Then, we prove that unimodular level curves of  rational inner functions are made up of smooth components. Section \ref{sec:CO} is dedicated to proving that facial contact order at a singularity can be read off by examining order of touching of generic unimodular level lines, a quantity we call \emph{order of contact}. This requires a careful analysis of Blaschke products arising from fixing one variable and viewing an RIF as a one-variable inner function in $\mathbb{D}$, together with a variational argument. In Section \ref{sec:COeq}, we prove that $z_1$- and $z_2$-contact orders of an RIF $\phi=\tilde{p}/p$ at a singular point are always equal, and we relate contact order to intersection multiplicity of $\mathcal{Z}_{p}$ and $\mathcal{Z}_{\tilde{p}}$ at a singularity.
Section \ref{sec:COfine} is devoted to a finer analysis of contact orders and order of contact. We exhibit a sophisticated generic mapping between branches of the zero set of the numerator of an RIF and the components of level curves of the associated RIF. In Section \ref{sec:construct} we present several different methods of constructing RIFs that allow us to prescribe properties of their zero sets, level lines, and singularities. Further examples that require more technical analysis or constructions from Section \ref{sec:construct}, or are related to finer points of our proofs, are discussed in Section \ref{sec:zoo}.

\section{Level sets near singularities} \label{sec:prelim}
\subsection{Preliminaries}
Let $\phi$ be an RIF on $\D^2$. As was mentioned in the Introduction, by \cite[Theorem $5.2.5$]{Rud69},
\[\phi(z_1, z_2)=\eta z_1^Mz_2^N\frac{\tilde{p}(z_1,z_2)}{p(z_1,z_2)},\]
where  $p$ is a polynomial of bidegree $(m,n)$ with no zeros in the bidisk, $M$ and $N$ are non-negative integers, $\tilde{p}(z_1, z_2):=z_1^mz_2^n\overline{p\big(\frac{1}{\bar{z}_1}, \frac{1}{\bar{z}_2}\big)}$
is the reflection of $p$, and $\eta$ is a unimodular constant. Without loss of generality, we can take $p$ to be atoral, so $p$ has at most finitely many zeros on $\mathbb{T}^2$, see \cite{AMS06}. As shown in \cite{Kne15}, $p$ also has no zeros on $(\D \times \T) \cup (\T \times \D).$  
As $\phi$ only has singularities at the zeros of $p$, it can have at most finitely many singularities on $\overline{\D^2}$ and these must all occur on $\mathbb{T}^2$. A monomial term will have little impact on the behavior of $\phi$ near a singular point and so, henceforth we will usually assume $\phi = \frac{\tilde{p}}{p}$ except in situations where the full characterization of RIFs is needed. 

Assume $\phi$ has a singularity at $\tau=(\tau_1, \tau_2) \in\mathbb{T}^2$.  We will study the local behavior of $\phi$ near such a singularity via two main objects:
\begin{itemize}
\item[1.] \textbf{The Zero Set of $\phi$.}  As $|\phi(z)|\le 1$ on $\mathbb{D}^2$, it follows that $p(\tau_1,\tau_2) = 0 = \tilde{p}(\tau_1,\tau_2).$ Thus, 
\[ \mathcal{Z}_{\tilde{p}}:= \{ (z_1, z_2) \in \mathbb{C}^2: \tilde{p}(z_1,z_2)=0\}\]
must have components passing through $\tau.$ In the first half of this preliminary section, we will parametrize such components of $\mathcal{Z}_{\tilde{p}}$ and precisely characterize the ways in which they can approach $\tau.$ \\

\item[2.] \textbf{The Unimodular Level Curves of $\phi$.}
For each $\lambda \in \mathbb{T}$, define 
\[ \mathcal{L}_{\lambda}(\phi) : = \left \{ (z_1, z_2) \in\mathbb{C}^2: \tilde{p}(z_1,z_2) = \lambda p(z_1,z_2) \right \}.\]
Then one can show (see Lemma \ref{lem:singularity}) that the \emph{level curve} $\mathcal{C}_{\lambda}:= \mathcal{L}_{\lambda}(\phi) \cap \mathbb{T}^2$ contains $\tau$ in its closure.  In the second half of this preliminary section, we obtain nice parametrizations of unimodular level curves and study how they pass through $\tau$.  

There is a special level curve associated with a singularity $\tau\in \T^2$ of $\phi$. Lemma $2.3$ in \cite{BPS17} gives a specific $\lambda_0 \in \mathbb{T}$ so that whenever $(z_n) \subseteq \mathbb{D}^2$ approaches $(\tau_1, \tau_2)$ nontangentially, $\phi(z_n)$ approaches $\lambda_0$: this number $\lambda_0$ will be referred to as the \emph{non-tangential value} of $\phi$ at the singularity $\tau$. We will call the level set $\mathcal{C}^*_{\lambda_0}(\tau)$ the \emph{value curve} of $\phi$ at $\tau=(\tau_1, \tau_2)$. 
\end{itemize}

In what follows, we will study the \emph{local} behavior of $\phi$ near a given singularity. Thus, without loss of generality, we will often make the following assumption:
\begin{center} (A1) \ \ Let $\phi = \frac{\tilde{p}}{p}$ be an RIF on $\mathbb{D}^2$ with a singularity at $(1,1)$ and associated $\lambda_0=1.$ 
\end{center}
It should be noted that if $\phi$ has multiple singularities 
on the two-torus, then each singularity has its own associated value curve. Away from its own singularity, a value curve usually exhibits the same features as any other level curve.
We shall frequently denote the value curve by $\mathcal{C}^*_{\lambda_0}$ when there is a unique singularity, or when it is clear from the context which singularity we are considering.
\subsection{Local Zero Set Behavior} 

\subsubsection{Parametrization.} 

As in \cite{BPS17}, we use Puiseux series to give local descriptions of $\mathcal{Z}_{\tilde{p}}$. To do this rigorously, we will need to transfer the problem to the upper half plane $\Pi$ via the following conformal map and its inverse:
\begin{equation}
 \beta \colon \Pi \rightarrow \mathbb{D}, \ \beta(w) := \frac{1+iw}{1-iw} \ \ \ \text{ and } \  \ \ \beta^{-1} \colon  \mathbb{D} \rightarrow \Pi, \ \beta^{-1}(z) := i \left[ \frac{1-z}{1+z} \right]. 
\label{eqn:beta}
\end{equation}
Then we can prove: 

\begin{theorem} \label{thm:zero} Assume $\phi$ satisfies (A1). Then there is an open set $\mathcal{V} \subseteq \mathbb{C}^2$ containing $(1,1)$ and positive integers $L_0, M_{1}, \dots, M_{L_0}$ such that the components of $\mathcal{Z}_{\tilde{p}} \cap \mathcal{V}$ can be described by the formulas
\begin{equation} \label{eqn:zero1}  z_1 = \psi_1^0(z_2), \ \dots, \ z_1 = \psi_{L_0}^0(z_2),\end{equation}
where the $\psi_{\ell}^0$ are obtained from convergent power series and have discontinuities only when $z_2=\beta(w_2)$ for  $w_2\in (-\infty, 0]$. 
\end{theorem}

\begin{proof} Let $\phi$ satisfy (A1) with $\deg p = (m,n)$ and define the polynomial 
\begin{equation} \label{eqn:q} q(w_1, w_2) := (1-iw_1)^m(1-iw_2)^n \tilde{p}\left (\beta(w_1), \beta(w_2)\right).\end{equation}
Then $q(0,0)=\tilde{p}(1,1) =0$. Moreover, as $\tilde{p}$ and $p$ possess no common factors, 
it follows that $q(0, w_2)$ is not identically $0$. 

Then Remark $3.4$ in \cite{BPS17} gives an open set $\mathcal{U} \subseteq \mathbb{C}^2$ containing $(0,0)$ where $\mathcal{Z}_{q}$ can be parameterized using Puiseux series. Specifically, all $(w_1, w_2) \in \mathcal{Z}_{q} \cap \mathcal{U}$ are given by the curves
\begin{equation} \label{eqn:qzero} w_1= \Psi^0_1\left(w_2^{\frac{1}{N_1}} \right), \dots,  w_1= \Psi_{L}^0 \left(w_2^{\frac{1}{N_L}}\right), \end{equation}
where $\Psi^0_1, \dots, \Psi^0_{L}$ are power series that converge in a neighborhood of $0$ each having $\Psi^0_{\ell}(0)=0$, the $L, N_1, \dots, N_{L}$ are positive integers, 
and for $w_2 \ne 0$, each term $w_2^{\frac{1}{N_{\ell}}}$ assumes $N_{\ell}$ separate values. Moreover, for $w_2$ sufficiently small, each $\big(  \Psi^0_{\ell}\big(w_2^{1/N_{\ell}} \big), w_2\big) \in \mathcal{Z}_q.$ 

Now set each $z_j = \beta(w_j)$. Then on an open set $\tilde{\mathcal{V}}\supseteq \overline{\mathbb{D}^2}$, we have $\tilde{p}(z_1, z_2) = 0$ if and only if $q(w_1, w_2)=0$. Define $\mathcal{V} = \beta(\mathcal{U})\cap  \tilde{\mathcal{V}}.$ Then $\mathcal{V} \subseteq \mathbb{C}^2$ is an open set containing $(1,1)$ and all $(z_1, z_2)$ in $\mathcal{Z}_{\tilde{p}} \cap \mathcal{V}$ are of the form
\begin{equation} \label{eqn:zero2} z_1= \beta\left( \Psi^0_1\left( \beta^{-1}(z_2)^{\frac{1}{N_1}} \right)\right), \dots,  z_1= \beta\left( \Psi_{L}^0\left( \beta^{-1}(z_2)^{\frac{1}{N_{L}}} \right)\right).\end{equation}
By fixing the standard branches of each  $\beta^{-1}(z_2)^{\frac{1}{N_{\ell}}}$ with discontinuities on $(-\infty, 0]$, we can alternately write $\mathcal{Z}_{\tilde{p}} \cap \mathcal{V}$ using $L_0 := N_1 + \dots + N_{L}$ formulas,
\[ z_1 = \psi_1^0(z_2),  \ \dots, \ z_1 = \psi_{L_0}^0(z_2).\]
For each $1 \le \ell \le L_0$, set $M_{\ell} = N_k$ where $\psi_{\ell}^0(z_2)=  \beta\big( \Psi^0_k\big( \beta^{-1}(z_2)^{\frac{1}{N_k}} \big) \big).$ Then each $\psi_{\ell}^0$  only has discontinuities when $z_2 =\beta(w_2)$ with $w_2\in (-\infty, 0]$. 
\end{proof}
\begin{remark} It is worth pointing out that the discontinuity mentioned in Theorem \ref{thm:zero} is somewhat artificial. It is a consequence of the fact that later we will need separate formulas for each piece or curve of $\mathcal{Z}_{\tilde{p}}.$ If instead, we studied the components of $\mathcal{Z}_{\tilde{p}}$ using the formulas in \eqref{eqn:zero2}, everything would appear continuous.

Note also that the branches of $\mathcal{Z}_{\tilde{p}}$ can only intersect a finite number of times near $z_2=1$ as the $\psi^0_j$ in \eqref{eqn:zero1} are algebraic functions.
\end{remark}

\subsubsection{Intersection Multiplicity.} \label{subsec:IM}

If $\phi$ has a singularity at $(1,1)$, then both $p$ and $\tilde{p}$ must vanish at $(1,1)$, so $(1,1)$ is an intersection point of $\mathcal{Z}_p$ and $\mathcal{Z}_{\tilde{p}}$.  The ``amount'' of intersection at a common zero $\tau$ of two polynomials $p$ and $q$ is called the \emph{intersection multiplicity} and is denoted $N_{\tau}(p, q).$ 

In this situation, $N_{(1,1)}(p, \tilde{p})$ can be computed  using the Puiseux series representations of $\mathcal{Z}_{q}$, as detailed in \cite[Appendix C]{Kne15}, where $q$ is the polynomial from \eqref{eqn:q}. In particular, transfer to $\Pi^2$ and factor $q=\alpha q_1 \cdots q_L$, where $\alpha$ is a unit and each $q_\ell$ is an irreducible Weierstrass polynomial in $w_1$ of degree $N_{\ell}$. Then define $\bar{q}(w_1, w_2) := \overline{q(\bar{w}_1, \bar{w}_2)}$, so  $\bar{q}=\bar{\alpha} \bar{q}_1 \cdots \bar{q}_L$ is a Weierstrass factorization of $\bar{q}$. Then the intersection multiplicity is:
\[ N_{(1,1)}(p, \tilde{p}) = N_{(0,0)}(q, \bar{q}) = \sum_{j=1}^L \sum_{k=1}^L N_{(0,0)} (q_j, \bar{q}_k),\]
where each $N_{(0,0)} (q_j, \bar{q}_k)$ is the order of  vanishing of the resultant
\[ \prod_{i=1}^{N_j} \prod_{\ell=1}^{N_k} \left(   \Psi^0_j\left ( \zeta^i t^{\frac{1}{N_j}}\right) -  \bar{\Psi}_k^0\left( \eta^\ell  t^{\frac{1}{N_k}} \right) \right),\]
where $\Psi^0_j$ and $\Psi^0_k$ are from \eqref{eqn:qzero} and $\zeta$ and $\eta$ are primitive $N_j^{th}$ and $N_k^{th}$ roots of unity respectively.  The arguments in \cite{Kne15} also show that 
 $N_{(1,1)}(p, \tilde{p})$ is even.  Moreover if $\deg p = (m,n)$, then  {B\'ezout's theorem} implies
\[N(p, \tilde{p}):=\sum_{\tau \in \mathcal{Z}_p\cap \mathcal{Z}_{\tilde{p}}}N_{\tau}(p,  \tilde{p})= 2mn,\]
and so in particular, the sum of the intersection multiplicities of common zeros of $p$ and $\tilde{p}$ on $\T^2$ is at most $2mn.$ See \cite{FulBook, Coxetal} for background and methods for computing intersection multiplicity. 

\subsubsection{Local Contact Order}

To see how $\mathcal{Z}_{\tilde{p}}$ approaches $(1,1)$, we require the following lemma:

\begin{lemma} \label{lem:LCO} Assume $\phi$ satisfies (A1) and has branches of $\mathcal{Z}_{\tilde{p}}$ given by \eqref{eqn:zero1}. Then for each branch $z_1  = \psi_{\ell}^0(z_2)$, there is an  
even number $\K^1_{\ell}$ so that 
\begin{equation} \label{eqn:CO} 1- |\psi_\ell^0(\zeta_2)| \approx |1 - \zeta_2|^{\K^1_\ell},\end{equation}
for all $\zeta_2 \in \mathbb{T}$ sufficiently close to $1$.  The number $\K^1_\ell$ is called the  \emph{$z_1$-contact order of the branch $z_1 = \psi_\ell^0(z_2)$}. Furthermore if,  $z_1 = \psi_{\ell}^0(z_2)$ and $z_1= \psi_{j}^0(z_2)$ are different branches of $\mathcal{Z}_{\tilde{p}}$ corresponding to the same $\Psi^0_{k}$ from \eqref{eqn:zero2}, then $\K^1_\ell = \K^1_j.$ 
\end{lemma}

\begin{proof} Assume $\phi$ satisfies (A1) and let $z_1 = \psi_\ell^0(z_2)$ be a branch of $\mathcal{Z}_{\tilde{p}}$ from Theorem \ref{thm:zero}. We can find $\K_\ell^1$ as in \eqref{eqn:CO}
using the proof of Theorem $3.5$ in \cite{BPS17}.  The basic idea is to switch to $\Pi^2$ and define $q$ as in \eqref{eqn:q}.  Then near $(0,0)$,  $\mathcal{Z}_{q}$ is described by the power series formulas in \eqref{eqn:qzero}. Let $\Psi^0_k$ denote the power series that gives rise to the specific branch $z_1 = \psi_\ell^0(z_2)$ via \eqref{eqn:zero2} and a choice of branch. Then, as $\Psi^0_{k}$ is a convergent power series around $0$ with $\Psi^0_k(0)=0$, we can write 
\[ \Psi^0_{k}(t) = \sum_{i=1}^{\infty} a_{ik} t^i,\]
for $t$ in a neighborhood $E \subseteq \mathbb{C}$ of $0$. By \cite[Theorem 3.3]{BPS17}, $t\mapsto (\Psi^0_k(t), t^{N_k})$ is injective into $\C^2\setminus \C^2_-$ near $(0,0)$. Then Lemma $18.3$ in \cite{Kne15} implies that there is an $M >0$ and constants $b_1, \dots, b_{2M-1} \in \mathbb{R}$ and $b_{2M} \in \mathbb{C}$ with $\Im(b_{2M}) >0$ so that 
\[ \Psi^0_{k}(t) = \sum_{i=1}^{2M-1} b_{i} t^{iN_{k}} + b_{2M} t^{2MN_{k}} + \sum_{i=2MN_{\ell}+1}^{\infty} a_{ik} t^{i}.\]
Then, following the arguments in the proof of \cite[Theorem $3.5$]{BPS17}, one can show that $\K^1_\ell =2M$.  This implies that $\K^1_\ell$ is even. Furthermore, this argument only depends on $
\Psi_{k}^0$. Thus, it shows that if   $z_1 = \psi_{\ell}^0(z_2)$ and $z_1= \psi_{j}^0(z_2)$ are  branches of $\mathcal{Z}_{\tilde{p}}$ corresponding to the same $\Psi^0_{k}$ (but different branches of 
$(\beta(z_2)^{-1})^{\frac{1}{N_{k}}})$, then their $z_1$-contact orders are equal.
\end{proof} 

\begin{remark} In Theorem \ref{thm:zero}, we could have instead described $\mathcal{Z}_{\tilde{p}}$ by writing $z_2$ in terms of $z_1$ like:
\[z_2 = \hat{\psi}_1^0(z_1), \ \dots, \ z_2 = \hat{\psi}_{J_0}^0(z_1). \] 
Then the \emph{$z_2$-contact order of each branch $z_2 = \hat{\psi}_j^0(z_1)$} is an even number $\K^2_{j}$ so that 
\[ 1- |\hat{\psi}_j^0(\zeta_1)| \approx |1 - \zeta_1|^{\K^2_j},\]
for all $\zeta_1 \in \mathbb{T}$ sufficiently close to $1$.  
\end {remark}

In \cite{BPS17}, we studied a global notion of $z_1$-contact order and used it to characterize the integrability of RIF derivatives. This global quantity can be recovered from the local quantities defined in Lemma \ref{lem:LCO}.  
\begin{definition} \label{def:CO1}  Let $\phi = \frac{\tilde{p}}{p}$ be a rational inner function on $\mathbb{D}^2$ with singularities $(\tau_1^1, \tau_2^1), \dots, (\tau_1^J, \tau_2^J)$  on $\mathbb{T}^2$. For each $1 \le j \le J$, one can apply  Lemma \ref{lem:LCO} to $\phi(\tau^j_1z_1, \tau^j_2z_2)$ to compute the contact 
order of the branches of $\mathcal{Z}_{\tilde{p}}$ near each $(\tau^j_1, \tau^j_2)$.  Then for $1\le j \le J$, let $\K^1_{(\tau^j_1, \tau^j_2)}$ be the maximum $z_1$-contact order of the branches of $\mathcal{Z}_{\tilde{p}}$ near $(\tau_1^j, \tau_2^j)$.  Then $\K^1_{(\tau^j_1, \tau^j_2)}$ is called \emph{the $z_1$-contact order of $\phi$ at $(\tau^j_1,\tau_2^j)$} and the \emph{global $z_1$-contact order of $\phi$} is given by
\[K_1 := \max \left\{\K^1_{(\tau^j_1, \tau^j_2)}: 1 \le j \le J \right\}.\]
\end{definition}
The quantity $K_1$ agrees with the definition in \cite{BPS17}. We also define analogous $z_2$-contact orders.

In \cite[Theorem 4.1]{BPS17}, we used global contact order to characterize integrability of derivatives of RIFs as follows:

\begin{theorem}  Let $\phi = \frac{\tilde{p}}{p}$ be an RIF on $\mathbb{D}^2.$ Then for $1 \le \p < \infty$, $\frac{\partial \phi}{\partial z_i} \in H^\p(\mathbb{D}^2)$ if and only if the $z_i$-contact order of $\phi$ satisfies $K_i < \frac{1}{\p-1}.$  
\end{theorem}

A modification of the arguments in \cite{BPS17} connects local derivative integrability with local contact order to yield the following:

\begin{theorem} \label{thm:integral} Let $\phi$ satisfy (A1). Then there is an open set $E_0 \subseteq \mathbb{T}^2$ containing $(1,1)$ so that for $1 \le \p < \infty$, and for all open $E \subseteq E_0$ containing $(1,1)$, the integral
\[ \iint_{E} \left | \tfrac{\partial \phi}{\partial z_i}(\zeta_1, \zeta_2) \right|^{\mathfrak{p}} |d\zeta_1| |d\zeta_2| < \infty  \]
 if and only if $\K^i_{(1,1)}< \frac{1}{\p-1}.$  
\end{theorem}

\begin{proof} As the proof is basically the same as that in \cite{BPS17}, with a restricted set of integration, we omit the details.
\end{proof}

\subsection{Unimodular Level Sets}

Let $\phi = \frac{\tilde{p}}{p}$ be an RIF on $\mathbb{D}^2$. Recall that for each $\lambda \in \mathbb{T}$, 
\[ \mathcal{L}_{\lambda}(\phi) : = \left \{ (z_1, z_2) \in\mathbb{C}^2: \tilde{p}(z_1,z_2) = \lambda p(z_1,z_2) \right \}\]
and $\mathcal{C}_{\lambda} : =  \mathcal{L}_{\lambda}(\phi) \cap \mathbb{T}^2$.
The connection between the singularities of $\phi$ and its unimodular level sets comes from a Hartogs principle via the Edge-of-the-Wedge Theorem in \cite{Pas17}. Specifically, the following is an immediate corollary of  \cite[Corollary 1.7]{Pas17}:

\begin{lemma} \label{lem:singularity} Assume $\phi = \frac{\tilde{p}}{p}$ is an RIF on $\mathbb{D}^2$ with a singularity at $(\tau_1, \tau_2) \in \mathbb{T}^2$. Then for each $\lambda \in \mathbb{T}$, the set
$\mathcal{C}_{\lambda}$ contains $(\tau_1, \tau_2)$ in its closure.
\end{lemma}

In what follows, we examine the way that components of a given $\mathcal{C}_{\lambda}$ approach the singular point $(1,1).$

\subsubsection{Smoothness.} Near the singular point $(1,1)$, each level set $\mathcal{L}_{\lambda}(\phi)$ is comprised of a union of smooth curves. The precise result is:

\begin{theorem} \label{thm:smooth} Let $\phi$ satisfy (A1) and fix $\mu \in \mathbb{T}$. Then there is a positive integer $L_{\mu}$, power series $\psi^{\mu}_1, \dots, \psi^{\mu}_{L_\mu}$ that converge in a neighborhood of $1$, and an open set $\mathcal{V} \subseteq \mathbb{C}^2$ of $(1,1)$
such that the components of $\mathcal{L}_{\mu}(\phi) \cap \mathcal{V}$ consists of sets described by the formulas
\begin{equation} \label{eqn:LL} z_1 = \psi^{\mu}_{1}(z_2), \dots, z_1 =  \psi^{\mu}_{L_\mu}(z_2),
\end{equation}
where $\psi^{\mu}_j(1)=1$ for $j=1, \ldots,  L_{\mu}$ and, for at most one value of $\mu$, possibly a straight line $\{z_2=1\}$.
 \end{theorem}

\begin{remark}
For $\lambda \notin \T$, an RIF level set $\mathcal{L}_{\lambda}(\phi)$ need not be smooth throughout $\C^2$. The rational inner function
\[\phi(z_1,z_2)=\frac{2z_1^2z_2^3-z_1^2-z_2^3}{2-z_1^2-z_2^3}\]
furnishes an example. We note that we have $\tilde{p}(0,0)=\frac{\partial \tilde{p}}{\partial z_1}(0,0)=\frac{\partial \tilde{p}}{\partial z_2}(0,0)=0$. The Puiseux parametrizations centered at $0$ in this case are of the form
\[z_1=f(z_2^{1/2})=\frac{z_2^{3/2}}{(2z_2^3-1)^{1/2}}=i z_2^{3/2}+\mathcal{O}(z_2^{2}),\]
and thus, $(0,0)\in \mathcal{L}_0(\phi)$ is a singular point that is not just a multiple point. 
\end{remark}

As in the previous section, we will use Puiseux series to parametrize the components of $\mathcal{L}_{\mu}(\phi) \cap \mathcal{V}$.  This step is encoded in the following lemma:

\begin{lemma} \label{lem:smooth1}  Let $r$ be a polynomial in $\mathbb{C}[w_1, w_2]$ with $r(0,0)=0$ and $r(w_1,0)$ not identically zero. Assume there is some neighborhood of $\widehat{\mathcal{U}} \subseteq \mathbb{C}^2$ of $(0,0)$ such that  $\mathcal{Z}_r \cap \left(\Pi^2 \cup \left ( -\Pi \right)^2 \right) \cap \widehat{\mathcal{U}} =\emptyset.$ Then there are power series $\Psi_1, \dots, \Psi_L$ that converge in a neighborhood of $0$ and an open set $\mathcal{U} \subseteq  \mathbb{C}^2$ containing $(0,0)$ such that $\mathcal{Z}_r \cap \mathcal{U}$ is described by the formulas 
\begin{equation} \label{eqn:zero3} w_1 = \Psi_{1}(w_2), \dots, w_1  = \Psi_L(w_2).\end{equation}
\end{lemma}

\begin{proof} Remark $3.4$ in \cite{BPS17} gives positive integers $L, N_1, \dots, N_L \in \mathbb{N}$, power series $\widehat{\Psi}_1, \dots, \widehat{\Psi}_L$ that converge near $0$ and satisfy $\widehat{\Psi}_{\ell}(0)=0$, and an open neighborhood $\mathcal{U} \subseteq \mathbb{C}^2$  of $(0,0)$ such that $\mathcal{Z}_r \cap \mathcal{U}$ is described by the formulas
\[ w_1 = \widehat{\Psi}_{1}\left(w_2^{\frac{1}{N_1}} \right), \dots, w_1  = \widehat{\Psi}_{L}\left(w_2^{\frac{1}{N_L}} \right),\]
where each $w_2^{1/N_{\ell}}$ is multi-valued.
Fix $\ell$ with $ 1\le \ell \le L.$ To simplify notation, define $\widehat{\Psi}:=\widehat{\Psi}_{\ell}$ and $N:= N_{\ell}$. Then, there are $a_k \in \mathbb{C}$ so that for $w \in \mathbb{C}$ near $0$, 
\[  \widehat{\Psi}(w) =\sum_{k=1}^{\infty} a_k w^k. \]
We claim $a_k$ can only be nonzero if $k$ is a multiple of $N$. To see this, fix a branch of $w^{\frac{1}{N}}$ so that if $t >0$, then $t^{\frac{1}{N}} = \zeta_n |t|^{\frac{1}{N}}$, where $\zeta_n$ is a fixed $N^{th}$ root of unity and $|t|^{\frac{1}{N}} >0$. Then for $t>0$ near $0$, we have
\[  \widehat{\Psi} \left (t^{\frac{1}{N}} \right) =\sum_{k=1}^{\infty} a_k \zeta_n^k |t|^{\frac{k}{N}} =\sum_{k=1}^{\infty} \Re\left( a_k \zeta_n^k \right) |t|^{\frac{k}{N}} + \sum_{k=1}^{\infty} \Im \left( a_k \zeta_n^k \right) |t|^{\frac{k}{N}}.  \]
We claim that $ \Im \big( a_k \zeta_n^k \big)=0$ for each $k \in \mathbb{N}.$  By way of contradiction, assume not and let $\tilde{k}$ be the smallest integer with  $ \Im \big( a_{\tilde{k}} \zeta_n^{\tilde{k}} \big) \ne 0$. Then for $t>0$ but near $0$, we have 
\[ \Im \left( \widehat{\Psi}\big (t^{\frac{1}{N}} \big) \right) \approx \Im \left( a_{\tilde{k}} \zeta_n^{\tilde{k}} \right) |t|^{\frac{\tilde{k}}{N}}.\]
By continuity, we can certainly find a $t_0>0$ with $\Im \big( \widehat{\Psi}\big (t_0^{\frac{1}{N}} \big)  \big) \ne 0$. Without loss of generality, assume $\Im \big(\widehat{\Psi}\big (t_0^{\frac{1}{N}} \big) \big)>0.$ As $\widehat{\Psi}(w_2^{\frac{1}{N}})$ is continuous near $t_0$, there must exist a $w_2 \in \mathbb{C}$ near $t_0$ with $\Im (w_2) >0$ and $\Im \big( \widehat{\Psi}\big(w_2^{\frac{1}{N}} \big) \big)>0$. By choosing $t_0$ sufficiently close to $0$, we can conclude that $r$ has a zero in $\Pi^2\cap \widehat{\mathcal{U}}$, a contradiction.

Thus, $ \Im \left( a_k \zeta_n^k \right)=0$ for each $k \in \mathbb{N}$ and for all $N^{th}$ roots of unity $\zeta_1, \dots, \zeta_N.$ This implies that each $a_k \in \mathbb{R}$ and if $a_k \ne 0$, then $k$ must be a multiple of $N$; namely, whenever $a_k \ne 0$, we can write $k =jN$ for some $j \in \mathbb{N}.$ This implies 
\[ \widehat{\Psi}\big(w_2^{\frac{1}{N}} \big)= \sum_{k=1}^{\infty} a_k w_2^{\frac{k}{N}} = \sum_{j=1}^{\infty} a_{jN} w_2^j. \]
Recalling the $\ell$-indices and defining $\Psi_{\ell}(w_2) =  \sum_{j=1}^{\infty} a_{jN_{\ell}} w_2^j$ gives the formulas in \eqref{eqn:zero3} and finishes the proof. \end{proof} 

Lemma \ref{lem:smooth1} has implications about the Weierstrass factorizations of such polynomials:

\begin{lemma}\label{lem:weier}
  Let $r \in \mathbb{C}[w_1,w_2]$ be  as in Lemma \ref{lem:smooth1}. Then each irreducible Weierstrass polynomial in $w_1$ in the Weierstrass factorization of $r$ is linear in $w_1$. \end{lemma}

\begin{proof} As discussed in \cite[Remark 3.4]{BPS17}, one can factor $r = \beta r_1 \cdots r_L$, where $\beta$ is a unit and each $r_{\ell}$ is an irreducible Weierstrass polynomial in $w_1$. Then as in the proof of  \cite[Theorem 3.3]{BPS17}, each Puiseux series describing $\mathcal{Z}_r$ originates as a description of the zero set of an $r_{\ell}$ and moreover, the denominator appearing in the fractional power of the Puiseux series gives the degree of $r_{\ell}$ in $w_1$. In the case of Lemma \ref{lem:smooth1}, the zero set components are given by analytic curves  $w_1= \Psi_{\ell}(w_2)$, which implies that each $\deg r_{\ell}=1$ in $w_1$.  So, the polynomials in the Weierstrass factorization of $r$ are all linear in $w_1.$
\end{proof}

An application of Lemma \ref{lem:smooth1} yields Theorem \ref{thm:smooth}:

\begin{proof} Set $p_{\mu}(z):= \tilde{p}(z) -\mu p(z)$. Then describing $\mathcal{L}_{\mu}(\phi)$ near $(1,1)$ is equivalent to describing $\mathcal{Z}_{p_{\mu}}$ near $(1,1)$.  Since $\phi$ is analytic and $|\phi|<1$ on $\mathbb{D}^2$, it is easy to see that $p_{\mu}$ has no zeros on $\mathbb{D}^2 \cup \mathbb{E}^2$, where $\mathbb{E}= \mathbb{C} \setminus \overline{\mathbb{D}}$ is the exterior disk. Assume $\deg p_{\mu} = (m,n)$ and define
\begin{equation} \label{eqn:Q2} q_{\mu}(w) := (1-iw_1)^m(1-iw_2)^n p_{\mu}(\beta(w_1), \beta(w_2)).\end{equation}
Since $\beta(0)=1$, we have $q_{\mu}(0,0) =0$ and since $\tilde{p}$ and $p$ share no common factors, $q_{\mu}(w_1, 0) \not \equiv 0$ for all 
but at most one $\mu\in \T$.

Suppose then that $q_{\mu}(w_1,0)\neq 0$. 
If $\widehat{\mathcal{U}} \subseteq \mathbb{C}^2$ is an open set containing $(0,0)$ that omits $w_1=-i$ and $w_2=-i$, then  $\mathcal{Z}_{q_\mu} \cap \left(\Pi^2 \cup (-\Pi)^2\right) \cap \widehat{\mathcal{U}} = \emptyset$. This means Lemma \ref{lem:smooth1} gives a positive integer $L_{\mu}$, power series $\Psi^{\mu}_1, \dots, \Psi^{\mu}_{L_\mu}$ that converge in a neighborhood of $0$, and an open set $\mathcal{U} \subseteq \widehat{\mathcal{U}}$ of $(0,0)$   such that $\mathcal{Z}_{q_{\mu}} \cap \mathcal{U}$ is described by the formulas 
\begin{equation} \label{eqn:LL2} w_1 = \Psi^\mu_{1}(w_2), \dots, w_1  = \Psi^\mu_{L_\mu}(w_2).\end{equation}
To describe $\mathcal{L}_{\mu}(\phi)$, recall that
\[  \mathcal{Z}_{q_{\mu}} \cap\widehat{\mathcal{U}}  =  \left\{ (w_1,w_2) \in \mathbb{C}^2: p_{\mu}(\beta(w_1), \beta(w_2))=0 \right \} \cap \widehat{\mathcal{U}} . \]
Setting $\mathcal{V} = \beta(\mathcal{U})$ and each $\psi^{\mu}_{\ell} = \beta \circ \Psi^{\mu}_{\ell} \circ \beta^{-1}$, we can switch variables via $z_1  = \beta(w_1)$ and $z_2 = \beta(w_2)$ to describe 
 the components of $\mathcal{L}_{\mu}(\phi) \cap \mathcal{V}$ with
\[ z_1 = \psi^{\mu}_{1}(z_2), \dots, z_1 =  \psi^{\mu}_{L_\mu}(z_2),\]
as needed.

If $q_{\mu}(w_1,0)$ vanishes identically, then $q_{\mu}(w_1,0)$ is divisible by $w_2$, and then tracing back we get a vertical component $\{z_2=1\}$ in $\mathcal{L}_{\mu}(\phi)$.
\end{proof}

\subsubsection{Horn Lemma} 
Assume $\phi$ satisfies (A1) and let $\mu \in \mathbb{T}$. Then Theorem \ref{thm:smooth} says the components of $\mathcal{L}_{\mu}(\phi)$ near $(1,1)$ are smooth curves, given by \eqref{eqn:LL}, and at most one vertical component. If we restrict attention to $\mathbb{T}^2$ and consider $\mathcal{C}_{\mu}:=\mathcal{L}_{\mu}(\phi) \cap \mathbb{T}^2$, these  smooth curves from \eqref{eqn:LL} approach $(1,1)$ within specific geometric regions. 

To simplify the geometry, we again perform our analysis on the upper half plane $\Pi$ and define
\[ \widetilde{\mathcal{C}}_{\mu} := \left \{ (x_1, x_2) \in \mathbb{R}^2: \tilde{p}(\beta(x_1), \beta(x_2)) = \mu p(\beta(x_1), \beta(x_2)) \right\}.\]
Then near $(0,0)$, we have  $\widetilde{\mathcal{C}}_{\mu} = \mathcal{Z}_{q_\mu} \cap \mathbb{R}^2$, where $q_\mu$ is from \eqref{eqn:Q2}. Near $(0,0)$, we also know $ \mathcal{Z}_{q_\mu}$  is described by \eqref{eqn:LL2} and each $\Psi_{\ell}^\mu$ is a convergent power series with real coefficients. Thus near $(0,0)$, $\widetilde{\mathcal{C}}_{\mu}$ is similarly described by the equations
\begin{equation} \label{eqn:LL3} x_1 = \Psi^{\mu}_1(x_2), \dots, \ x_1 = \Psi^{\mu}_{L_\mu}(x_2).\end{equation}

We can slightly modify some ideas from \cite{BPS17} to show that the curves in \eqref{eqn:LL3}
approach $(0,0)$ in a specific way:

\begin{lemma} \textbf{(A Horn Lemma.)} \label{lem:horn} Let $\phi$ satisfy (A1) and fix $\mu \in \mathbb{T}$ with $\mu \ne 1$.  Then for each $\Psi^{\mu}_{\ell}$ in \eqref{eqn:LL3}, there is an $m_{\ell} <0$ and $b_{\ell} >0$ so that 
\begin{equation} \label{eqn:horn} \frac{x_2}{m_{\ell} + b_{\ell}x_2} < \Psi^{\mu}_{\ell}(x_2) < \frac{x_2}{m_{\ell} - b_{\ell}x_2}\end{equation}
for $x_2 \in \mathbb{R}$ sufficiently close to $0$.
\end{lemma}

\begin{proof} Fix $\ell$ and consider the curve $z_1 = \psi^\mu_{\ell}(z_2)$ restricted to $\mathbb{T}^2$ from \eqref{eqn:LL}. Change variables to $\Pi^2$ by defining each $\tilde{x}_j = \alpha(z_j)$, where $\alpha(z) := i \left( \frac{1+z}{1-z}\right)$ is a conformal map from $\mathbb{D}$ to $\Pi$. This gives a new curve
\[ \tilde{x}_1 =\left( \alpha \circ \psi^{\mu}_{\ell} \circ \alpha^{-1}\right) (\tilde{x}_2)\] 
in $\mathbb{R}^2$ that approaches $(\infty, \infty)$. 
As $\mu \ne 1$, the arguments in \cite[Proposition 5.5]{BPS17} can be used to show that this curve approaches $(\infty, \infty)$ within a ``spoke region'' associated to a Pick function $f$ defined using $\phi$.

Change variables again by defining each $x_j = \gamma(\tilde{x}_j)$, where $\gamma(w):= -\frac{1}{w}$ conformally maps $\Pi$ to $\Pi$. This gives the curve of interest:
\[ x_1 = \left(\gamma \circ  \alpha \circ \psi^{\mu}_{\ell} \circ \alpha^{-1} \circ \gamma^{-1} \right) (x_2) = \left( \beta^{-1} \circ \psi^{\mu}_{\ell} \circ \beta \right) (x_2) = \Psi^{\mu}_{\ell} (x_2),\]
as in \eqref{eqn:LL3}.
Then the arguments in \cite[Lemma $5.7$]{BPS17} imply that this curve approaches $(0,0)$ within a ``horn region'' as shown below in Figure \ref{fig:horn}. This means that there are constants $m_{\ell} <0$ and $b_{\ell} >0$ so that \eqref{eqn:horn} holds for $x_2$ near $0$. To avoid a lengthy discussion of Pick functions and spoke regions, we omit further details and refer the reader to \cite{ATDY16, BPS17}.
\end{proof}

\begin{center}
\begin{figure}[!htbp]
  \includegraphics[scale=1.3]{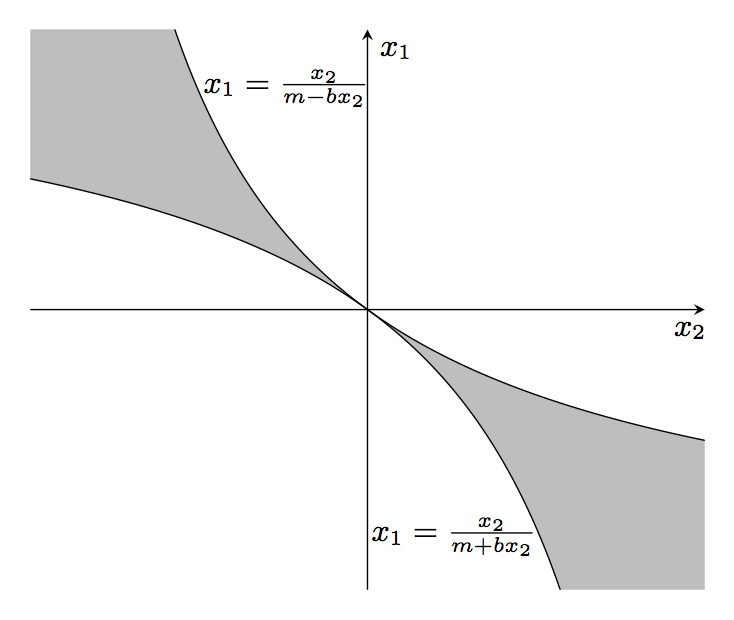}
  \caption{A Horn Region near $(0,0)$}
  \label{fig:horn}
\end{figure}
\end{center}

This has implications about the power series representations for each  function $\psi^{\mu}_{\ell}$ from Theorem \ref{thm:smooth}. Specifically:

\begin{lemma} \label{lem:hornlinearterm} Let $\phi$ satisfy (A1). Then for $\mu\neq 1$, the power series representation of each  $\psi^{\mu}_{\ell}$ from \eqref{eqn:LL} centered at $1$ has a nonzero linear term.\end{lemma}

 \begin{proof} Fix any  $\psi^{\mu}_{\ell}$ from \eqref{eqn:LL}.  Then $\psi^{\mu}_{\ell}$ satisfies $ \psi^{\mu}_{\ell} = \beta \circ \Psi^{\mu}_{\ell} \circ \beta^{-1}$, for some $ \Psi^{\mu}_{\ell}$ from \eqref{eqn:LL2}. Denote the power series of $ \Psi^{\mu}_{\ell}$ centered at $0$ by:
\[ \Psi^{\mu}_{\ell}(w_2) = \sum_{k=1}^{\infty} a_{k \ell} w_2^k.\]
Then Lemma \ref{lem:horn} immediately implies that for $x_2 \in \mathbb{R}$ near $0$, we have $|\Psi^{\mu}_{\ell}(x_2)| \approx |x_2|$ and so $a_{1\ell} \ne 0$.  Equivalently, $\left( \Psi^{\mu}_{\ell} \right)'(0) \ne 0.$  As $ \psi^{\mu}_{\ell} = \beta \circ \Psi^{\mu}_{\ell} \circ \beta^{-1}$,
it follows that $\left( \psi^{\mu}_{\ell}\right)' (1) \ne 0.$
Thus, the power series representation of  $\psi^{\mu}_{\ell}$ centered at $1$ has a nonzero linear term.
\end{proof}

\subsubsection{Order of Contact}
Let $\phi$ satisfy (A1) and fix $\lambda, \mu \in \mathbb{T}.$ Excepting at most one $\mu$, there are components of $ \mathcal{L}_{\lambda}(\phi)$ and $ \mathcal{L}_{\mu}(\phi)$, given by the analytic curves in \eqref{eqn:LL}, which approach $(1,1)$. To analyze the relationship between the branches of  $ \mathcal{L}_{\lambda}(\phi)$ and $ \mathcal{L}_{\mu}(\phi)$ near $(1,1)$, we define the following:

\begin{definition}\label{def:branchOC} 
Assume $z_1 = \psi_1(z_2)$ and $z_1= \psi_2(z_2)$ are analytic curves defined in a neighborhood of $z_2=d$ with $\psi_1(d) =c= \psi_2(d).$ Then, the {\it order of contact} of $z_1 = \psi_1(z_2)$ and $z_1= \psi_2(z_2)$ at the point $(c, d)$ is the smallest positive integer $K$ with 
\[ 
\left  | \psi_1(z_2) -  \psi_{2}(z_2) \right | \approx  |d-z_2|^{K},\]
for $z_2$ near $d$. Equivalently, by examining the power series representations centered at $d$, one can show that $K$ is the smallest positive integer satisfying the derivative condition:
\[  \psi_{1} ^{(K)}(d) \ne \psi_{2}^{(K)}(d). \]
\end{definition}

In particular, we will study the order of contact at $(1,1)$ between branches $z_1 =  \psi^{\lambda}_{i}(z_2)$ and $z_1 =  \psi^{\mu}_{j}(z_2)$ of $ \mathcal{L}_{\lambda}(\phi) $ and $ \mathcal{L}_{\mu}(\phi) $ respectively. We further define

\begin{definition} Let $\phi$ satisfy (A1) and fix $\lambda, \mu \in \mathbb{T}$ with $\lambda \neq \mu$. Then the \emph{$z_1$-order of contact between $\mathcal{L}_{\lambda}(\phi)$ and $\mathcal{L}_{\mu}(\phi)$}, denoted $\K^{\lambda,\mu}_{(1,1)}$, is the maximum order of contact at $(1,1)$ between any two branches $z_1 =  \psi^{\lambda}_{i}(z_2)$ and $z_1 =  \psi^{\mu}_{j}(z_2)$ of $ \mathcal{L}_{\lambda}(\phi) $ and $ \mathcal{L}_{\mu}(\phi) $ from \eqref{eqn:LL}.
\end{definition}

Finally, we observe that order of contact is invariant under our typical change of variables. Specifically, recall that each $\psi^{\mu}_{\ell}$ from Theorem \ref{thm:smooth} satisfies $ \psi^{\mu}_{\ell} = \beta \circ \Psi^{\mu}_{\ell} \circ \beta^{-1}$, where $w_1 = \Psi^{\mu}_{\ell} (w_2)$ is a component of $\mathcal{Z}_{q_{\mu}}$ from \eqref{eqn:LL2}. Then the order of contact at $(0,0)$ between each $w_1 = \Psi^{\lambda}_{i}(w_2)$ and $w_1 =  \Psi^{\mu}_{j}(w_2)$
must equal the  order of contact at $(1,1)$ between the associated curves $z_1 =  \psi^{\lambda}_{i}(z_2)$ and $z_1 =  \psi^{\mu}_{j}(z_2)$.

\section{Contact order vs order of contact}\label{sec:CO}

In this section, we assume $\phi$ satisfies (A1) and then reconcile our two competing notions of contact order. Specifically, we consider the  the $z_1$-contact order of $\phi$ at $(1,1)$, which measures how the zero set of $\phi$ approaches $(1,1)$ and the $z_1$-order of contact between unimodular level curves of $\phi$, which measures the amount of similarity between unimodular level curves of $\phi$ near $(1,1).$  Here is the precise result:

\begin{theorem} \label{thm:CO} Let $\phi$ satisfy (A1).  Then for every pair $\lambda, \mu \in \mathbb{T}$, excluding 
at most one $\mu_0,$   the $z_1$-contact 
order of $\phi$ at $(1,1)$ equals the $z_1$-order of contact 
between the unimodular level curves  $\mathcal{L}_{\lambda}(\phi)$ and $\mathcal{L}_{\mu}(\phi)$ at $(1, 1)$. \end{theorem}

\begin{definition} Let $\mu_0$ denote the excluded value from Theorem \ref{thm:CO}. Then the level set $\mathcal{C}^{**}_{\mu_0}$ is called the \emph{exceptional level curve} at $(1,1).$ Level curves that are neither value curves nor exceptional curves are called \emph{generic}. 
\end{definition}
In many cases, we have $\mu_0=\lambda_0$, so that the value curve and the exceptional curve are one and the same. However, this is not always the case. In Example \ref{nonvalueexcept} we use our methods for constructing RIFs with prescribed properties to exhibit an RIF with an exceptional curve that does not coincide with the value curve.

To prove Theorem \ref{thm:CO}, we will require preliminary information about finite Blaschke products and their behavior on arcs $A \subseteq \mathbb{T}.$ 

\subsection{Movements of Blaschke Products}

First, recall  \cite[Lemma $4.2$]{BPS17}:

\begin{lemma} \label{lem:blaschke} Consider a finite Blaschke product $b(z) := \prod_{j=1}^n b_{\alpha_j}(z), \text{ with } b_{\alpha_j}(z) = \frac{z-\alpha_j}{1-\bar{\alpha}_j z}$ for $\alpha_j \in \mathbb{D}.$
Then the modulus of the derivative of $b$ satisfies
\begin{equation} \label{eqn:derivnorm}  |b'(\zeta)| = \frac{ b'(\zeta)}{b(\zeta)} \zeta = \sum_{j=1}^n \left |b'_{\alpha_j}(\zeta) \right | \qquad \text{ for } \zeta \in \mathbb{T}. \end{equation}
\end{lemma}

Given Lemma \ref{lem:blaschke}, the following definition makes sense:

\begin{definition} Let $b(z) := \prod_{j=1}^n b_{\alpha_j}(z), \text{ with } b_{\alpha_j}(z) = \frac{z-\alpha_j}{1-\bar{\alpha}_j z}$ for $\alpha_j \in \mathbb{D}.$ Then the \emph{movement of $b$} is the measure $\mu_b$ on $\mathbb{T}$ defined by 
\begin{equation} \label{eqn:mub} \mu_b(A) :=\int_A \frac{ b'(\zeta)}{b(\zeta)} \zeta |d \zeta| =   \int_A |b'(\zeta)| |d \zeta| = \sum_{j=1}^n \int_A  \left |b'_{\alpha_j}(\zeta) \right | |d \zeta|,  \ \ \text{ for  measurable } A \subseteq \mathbb{T},\end{equation}
where $|d\zeta|$ denotes Lebesgue measure on $\T$.
\end{definition}

In what follows, we will need two ways to denote the length of arcs in $\mathbb{T}$. First, given a standard arc $A\subseteq \mathbb{T}$, we let $|A|$ denote the length, or Lebesgue measure, of $A$. Similarly given an arc $A$ that winds around $\mathbb{T}$, we let $|A|_{\mathcal{W}}$ denote the length of the curve taking the winding (or multiplicity) into account. For example, if $b$ is a finite Blaschke product with $\deg b =n$, then $b(\mathbb{T})$ is a curve winding around the torus $n$ times, so $| b(\mathbb{T})|_{\mathcal{W}} = 2\pi n.$

The following lemma details the needed properties of $\mu_b$:

\begin{lemma} \label{lem:move} For each finite Blaschke product $b$, define $\mu_b$ as in \eqref{eqn:mub}. Then these measures satisfy the following properties:
\begin{itemize}
\item[A.] If $b_1$ and $b_2$ are finite Blaschke products and if $A \subseteq \mathbb{T}$, then $\mu_{b_1 b_2}(A) = \mu_{b_1}(A) + \mu_{b_2}(A).$
\item[B.] If $A$ is an arc in $\mathbb{T}$, then $\mu_b(A) =| b(A)|_{\mathcal{W}}.$ Specifically, if $\deg b =n$, then $\mu_b(\mathbb{T}) = 2\pi n.$ 
\item[C.] For each $\alpha \in \mathbb{D}$ and $\epsilon >0$, there is an arc $A_{\epsilon, \alpha} \subseteq \mathbb{T}$ centered at $\frac{\alpha}{|\alpha|}$ such that
\begin{itemize}
\item[i.] $\mu_{b_{\alpha}} (A_{\epsilon, \alpha}) > 2\pi -\epsilon$;
\item[ii.] $\left | A_{\epsilon, \alpha}\right| \le c_{\epsilon} (1-|\alpha|)$, where $c_{\epsilon}>0$ is a constant independent of $\alpha.$ 
\end{itemize}
\end{itemize}
\end{lemma}

\begin{proof} Property $A$ follows immediately from the fact (implied by Lemma \ref{lem:blaschke}) that $\left| \frac{d}{dz}(b_1 b_2) \right| = |b_1'| + |b_2'|$ on $\mathbb{T}.$ Property $B$ follows from the Argument Principle. 
To prove Property $C$, fix $\epsilon >0$. Choose $k>0$ large enough so that 
\[ 4 \tan^{-1} \left(\frac{k}{2}\right) > 2\pi -\epsilon.\]
Set $c_{\epsilon} = 4k.$ Choose $\alpha \in \mathbb{D}$ and without loss of generality, assume $\alpha =t \geq 0.$ Then we have two cases.
\begin{itemize}
\item[Case 1:] If $k(1-t) \ge \frac{\pi}{2}$, then choose $A_{\epsilon, t} = \mathbb{T}.$ This immediately gives:
\[ \mu_{b_t}(A_{\epsilon, t}) = 2\pi > 2\pi -\epsilon \ \text{ and } \
\left| A_{\epsilon, t}\right| = 2 \pi \le 4 k (1-t) = c_{\epsilon} (1-t),\]
as needed.

\item[Case 2:] If $k(1-t) < \frac{\pi}{2}$, then choose $A_{\epsilon, t}$ to be the arc in $\mathbb{T}$ with points $e^{i\theta}$ corresponding to $\theta \in [-(1-t)k, (1-t)k].$ Then $A_{\epsilon, t}$ is centered at $\frac{t}{|t|} =1$ and  with $c_\epsilon$ as above,
\[ \left| A_{\epsilon, t}\right| = 2 k(1-t) \le c_{\epsilon} (1-t).\]
Similarly, we can compute
\[ \begin{aligned}
\mu_{b_t}(A_{\epsilon,t}) &= \int_{A_{\epsilon, t}} |b'_t(\zeta)| |d\zeta| \\
&= \int_{-k(1-t)}^{k(1-t)} \frac{1-t^2}{1-2t\cos \theta +t^2} d \theta \\
& = 4 \tan^{-1}\left( \frac{ (t+1)\tan(\frac{k}{2}(1-t))}{1-t}\right) \\
&\ge 4 \tan^{-1}\left( \frac{\tan(\frac{k}{2}(1-t))}{1-t}\right)\\
& \ge 4 \tan^{-1} \left(\frac{k}{2}\right) \\
& > 2\pi -\epsilon,
\end{aligned}
\]
where we used the fact that $\tan^{-1}(x)$ is increasing and $\tan x \ge x$ for $0 \le x < \frac{\pi}{2}.$
\end{itemize}
\end{proof}


\subsection{Proof of Theorem \ref{thm:CO}}
To prove Theorem \ref{thm:CO}, we first show that for all $\lambda, \mu \in \mathbb{T}$, excepting one $\mu_0$, there is some pair of branches of  $\mathcal{L}_{\lambda}(\phi)$ and $\mathcal{L}_{\mu}(\phi)$ whose $z_1$-order of contact at $(1,1)$ is at least the  $z_1$-contact  order of $\phi$ at $(1,1)$, denoted $\K^1_{(1,1)}$. Specifically:

\begin{theorem} \label{thm:CO1} Let $\phi$ satisfy (A1). Then, given any pair $\lambda, \mu \in \mathbb{T}$, excluding  at most one $\mu_0,$  there are branches of 
$\mathcal{L}_{\lambda}(\phi)$ and $\mathcal{L}_{\mu}(\phi)$ whose $z_1$-order of contact at $(1, 1)$ is at least $\K^1_{(1,1)}.$ 
\end{theorem}

\begin{proof} By definition, there is at least one branch of $\mathcal{Z}_{\tilde{p}}$ near $(1,1)$ whose $z_1$-contact order is $\K^1_{(1,1)}.$ Fix 
such a branch and call it $z_1=\psi_0(z_2)$, and fix
any $\zeta \in \mathbb{T}\setminus\{1\}$ near $1$. Then $\phi_{\zeta}(z_1):= \phi(z_1, \zeta)$ is a finite Blaschke product and $\phi(  \psi_0(\zeta), \zeta)=0.$ Thus, the Blaschke product $b_{\psi_0(\zeta)}$ is a factor of $\phi_{\zeta}.$ 

Let $(\zeta_n) \subseteq \mathbb{T}$ be a sequence converging to $1$, with each $\zeta_n\neq 1$. Fix $\epsilon >0$ and for each $n$, let $A_{\epsilon}^n := A_{\epsilon, \psi_0(\zeta_n)}$ denote the arc from Lemma \ref{lem:move}. Define the image set
\[ I_{\epsilon}^n :=\left \{  \phi(\tau, \zeta_n) : \tau \in  A_{\epsilon}^n \right \},\]
where points are counted according to multiplicity.
Then as $\phi_{\zeta_n}$ is continuous on $\mathbb{T}$, we know $ I_{\epsilon}^n$ is an arc winding around $\mathbb{T}$ and by Lemma \ref{lem:move},
\[ \left | I_{\epsilon}^n \right|_{\mathcal{W}} = \mu_{\phi_{\zeta_n}}\left(A_{\epsilon}^n \right)  \ge \mu_{b_{\psi_0(\zeta_n)}}(A_{\epsilon}^n ) > 2\pi -\epsilon,\]
where $\left | \cdot \right|_{\mathcal{W}}$ indicates the length of an arc winding around $\mathbb{T}$.  Then each $I_{\epsilon}^n$ contains an arc, call it  $T^n_{\epsilon}$, composed of distinct points in $\mathbb{T}$ with length $2\pi -\epsilon$.  Let $c_n$ denote the center of $T^n_{\epsilon}$. By passing to a subsequence, we can assume that the sequence $(c_n)$ converges to some $c \in \mathbb{T}.$  Let $B_{\epsilon}$ denote the arc contained in $\mathbb{T}$ with center $c$ and length $\left | B_{\epsilon}\right | = 2\pi - 2\epsilon.$
Then if we choose $N$ sufficiently large, we will have $B_{\epsilon} \subseteq T^n_{\epsilon}$ for $n \ge N.$

Now, fix any $\lambda, \mu \in B_{\epsilon}.$  We claim there are branches from \eqref{eqn:LL} of the level sets 
$\mathcal{L}_{\lambda}(\phi)$ and $\mathcal{L}_{\mu}(\phi)$ whose order of contact at $(1, 1)$ is at least $\K^1_{(1,1)}.$  By Theorem \ref{thm:smooth}, the branches of $ \mathcal{L}_{\mu}(\phi) $ (and similarly of $\mathcal{L}_{\lambda}(\phi))$ near $(1,1)$ are given by smooth curves: 
\[ z_1 =  \psi^{\mu}_{1}(z_2),  \dots, \ z_1= \psi^{\mu}_{L_\mu}(z_2),
\] 
and possibly the straight line $\{z_2=1\}$.
Then for $n$ sufficiently large, $\lambda, \mu \in T^n_{\epsilon}$ and so, there must be points
$\tau_n, \eta_n \in A_\epsilon^n$ with $\phi(\tau_n, \zeta_n) = \lambda$ and $\phi(\eta_n, \zeta_n) = \mu.$ 
Since $\zeta_n\neq 1$, as long as $n$ is large enough, there will also be indices $i_n, j_n$ 
so that $\tau_n = \psi^{\lambda}_{i_n}(\zeta_n)$ and $\eta_n = \psi^{\mu}_{j_n}(\zeta_n)$, so 
$\psi^{\lambda}_{i_n}(\zeta_n),   \psi^{\mu}_{j_n}(\zeta_n) \in A^n_{\epsilon}.$ By passing to a subsequence, we can assume that the points
  $\psi^{\lambda}_{i_n}(\zeta_n)$ and $\psi^{\mu}_{j_n}(\zeta_n)$ all come from the 
  same branches of $\mathcal{L}_{\lambda}(\phi)$ and $\mathcal{L}_{\mu}(\phi)$ respectively, 
  say from $z_1 =  \psi^{\lambda}_{i}(z_2)$ and  $z_1=\psi^{\mu}_{j}(z_2)$. Then Lemma \ref{lem:move} implies
\[
\left |  \psi^{\lambda}_{i}(\zeta_n)- \psi^{\mu}_{j}(\zeta_n) \right | \le \left | A_{\epsilon}^n \right | 
 \le c_{\epsilon} \left( 1 - | \psi_0(\zeta_n) |\right) 
\approx \left | 1 - \zeta_n \right |^{\K^1_{(1,1)}},
\]
for $n$ sufficiently large. Then, the smoothness of $z_1 =  \psi^{\lambda}_{i}(z_2)$ and  $z_1=\psi^{\mu}_{j}(z_2)$ implies that their $z_1$-order of contact at $(1,1)$ is at least $\K^1_{(1,1)}.$

Finally, we claim that for all $\mu, \lambda \in \mathbb{T}$, except for possibly one $\mu_0 \in \mathbb{T}$, there is an $\epsilon>0$ so that $\mu, \lambda \in B_{\epsilon}.$ First
if for every $\mu, \lambda \in \mathbb{T}$, there is a $B_{\epsilon}$ containing $\mu, \lambda$, we are done. So, assume there is some pair 
 $\mu_0, \lambda_0 \in \mathbb{T}$ with no common $B_{\epsilon}.$ We will show that this cannot happen for any other $\mu$.  
By assumption, each $B_{\epsilon}$ must omit a small arc containing $\mu_0$ or a small arc containing $\lambda_0$. By switching $\mu_0$ and $\lambda_0$ if necessary, we can find a sequence $\epsilon_m \rightarrow 0$ such that each $B_{\epsilon_m}$ omits only an interval of length $2\epsilon_m$ containing $\mu_0$. 
Then for every other pair $\mu, \lambda \in \mathbb{T}$ with neither equal to $\mu_0$, there will be some $\epsilon_m>0$ with $\mu, \lambda  \in B_{\epsilon_m}$, as needed. 

Thus, for each $\lambda, \mu \in \mathbb{T}$, except possibly one $\mu_0$, we can apply our earlier arguments and obtain  branches of $\mathcal{L}_{\lambda}(\phi)$ and $\mathcal{L}_{\mu}(\phi)$ whose $z_1$-order of contact at $(1, 1)$ is at least $\K^1_{(1,1)}.$
\end{proof}

Now we show the converse:

\begin{theorem} \label{thm:CO2} Let $\phi$ satisfy (A1). Then, given any pair $\lambda, \mu \in \mathbb{T}$,  the $z_1$-order of contact of $\mathcal{L}_{\lambda}(\phi)$ and $\mathcal{L}_{\mu}(\phi)$ at $(1, 1)$ cannot exceed  $\K^1_{(1,1)}.$ 
\end{theorem}

\begin{proof} By way of contradiction, assume there are branches $z_1 =  \psi^{\lambda}_{i}(z_2)$ and $z_1 =   \psi^{\mu}_{j}(z_2)$ of  $\mathcal{L}_{\lambda}(\phi)$ and $\mathcal{L}_{\mu}(\phi)$ from \eqref{eqn:LL} with order of contact $\K > \K_{(1,1)}^1$.  By Theorem \ref{thm:integral}, there is a neighborhood $E \subseteq \mathbb{T}^2$ of $(1,1)$ such that 
\begin{equation} \label{eqn:integral} \iint_{E} \left | \tfrac{\partial \phi}{\partial z_1}(\zeta_1, \zeta_2) \right|^{\mathfrak{p}} |d\zeta_1| |d\zeta_2| = \infty  \end{equation}
 if and only if $\p \ge \frac{1}{\K^1_{(1,1)}} +1.$

Define $\Phi:= \phi \circ \beta$ and let  $w_1 =  \Psi^{\lambda}_{i}(w_2)$ and $w_1 =   \Psi^{\mu}_{j}(w_2)$ denote the corresponding smooth level curves of $\Phi$ near $(0,0),$ as given in \eqref{eqn:LL2}. Then they also have order of contact $\K$ at $(0,0)$. For $x_2$ sufficiently small and positive, say $0<x_2<a$, we know that $ \Psi^{\lambda}_{i}(x_2)- \Psi^{\mu}_{j}(x_2)$ does not change sign. Then without loss of generality, we can assume  $\Psi^{\lambda}_{i}(x_2) <\Psi^{\mu}_{j}(x_2)$ on $[0,a].$ Define 
\[ \Omega := \left \{ (x_1, x_2): x_2 \in [0,a] \text{ and } x_1 \in [  \Psi^{\lambda}_{i}(x_2), \Psi^{\mu}_{j}(x_2)] \right \}.\]
If we choose $a$ sufficiently small, then arguments identical to those in the proof of \cite[Lemma 5.8]{BPS17} imply that if 
\[  \iint_{\Omega} \left | \tfrac{ \partial \Phi}{\partial x_1} (x_1,x_2) \right|^\p  dx_1  dx_2 = \infty, \ \text{ then }  \ \iint_{E} \left | \tfrac{\partial \phi}{\partial z_1} (\zeta_1, \zeta_2) \right|^\p |d \zeta_1| |d\zeta_2| = \infty,\]
for $0 < \p <\infty$. Now we use variational arguments analogous to those in the proof of \cite[Proposition 5.9]{BPS17}. Specifically, fix $x_2 \in [0,a].$ Then the Euler-Lagrange equations can be used to show
\[
 \int_{ \Psi^{\lambda}_{i}(x_2)}^{ \Psi^{\mu}_{j}(x_2)}  \left | \tfrac{ \partial \Phi}{\partial x_1} (x_1,x_2) \right|^\p  dx_1 \ge \frac{\left | \Phi( \Psi^{\mu}_{j}(x_2), x_2) - \Phi(  \Psi^{\lambda}_{i}(x_2), x_2)\right|^{\p}}{ |  \Psi^{\lambda}_{j}(x_2)-  \Psi^{\mu}_{i}(x_2)|^{\p-1}}
\approx |x_2|^{\K(1-\p)}.
\]
From this, we have 
\[ \iint_{\Omega} \left | \tfrac{ \partial \Phi}{\partial x_1} (x_1,x_2) \right|^\p  dx_1  dx_2  \gtrsim \int_0^a  |x_2|^{\K(1-\p)} dx_2 = \infty\]
if $\K(1-\p) \le -1$ or equivalently $\p \ge \frac{1}{\K} +1.$  But, this implies 
$\eqref{eqn:integral} = \infty$ for $\p \ge  \frac{1}{\K} +1,$ which is a strictly larger class of $\p$ than those satisfying $\p \ge \frac{1}{\K^{1}_{(1,1)}} +1$, a contradiction.
\end{proof}

\section{Equal contact orders}\label{sec:COeq}

Throughout \cite{BPS17} and in Section \ref{sec:prelim} of this paper, we discussed both the $z_1$- and $z_2$-contact orders of an RIF at a singularity. Perhaps surprisingly, the results of Section \ref{sec:CO} show that these two quantities are equal. 

\begin{theorem} \label{thm:ECO}Assume $\phi$ satisfies (A1). Then $\K^1_{(1,1)}  = \K^2_{(1,1)}.$
 \end{theorem}
 
\begin{proof} We first show $\K^2_{(1,1)}  \le \K^1_{(1,1)}.$ By Theorem \ref{thm:CO}, there are $\lambda, \mu \in \mathbb{T}$ and branches $z_1 =  \psi^{\lambda}_{i}(z_2)$ and  $z_1=\psi^{\mu}_{j}(z_2)$ of  the level sets 
$\mathcal{L}_{\lambda}(\phi)$ and $\mathcal{L}_{\mu}(\phi)$ so that 
\[ \left | \psi^{\lambda}_{i}(z_2) -  \psi^{\mu}_{j}(z_2) \right | \approx  |z_2-1|^{\K^1_{(1,1)}},\]
for $z_2$ near $1$. By Theorem \ref{thm:smooth},  $\psi^{\lambda}_{i}$  and  $\psi^{\mu}_{i}$ have power series expansions at $1$ as follows:
\[ \psi^{\lambda}_{i}(z_2) = \sum_{k=1}^{\infty} a_k (z_2-1)^k \ \ \text{ and }  \ \ \psi^{\mu}_{j}(z_2) = \sum_{k=1}^{\infty} b_k (z_2-1)^k.\]
By Lemma \ref{lem:hornlinearterm}, we have $a_1, b_1 \ne 0.$ Then by the Lagrange inversion formula, we can write 
\[ z_2 = \left(  \psi^{\lambda}_{i} \right)^{-1} (z_1) = \sum_{k=1}^{\infty} g_k(1) \frac{(z_1-1)^k}{k!} \quad \text{ where } g_k(1) = \lim_{w \rightarrow 1} \left( \frac{d^{k-1}}{dw^{k-1}}\left( \frac{w-1}{\psi^{\lambda}_i(w) - \psi^{\lambda}_i(1)} \right)^k \right),\]
as a convergent power series around $z_1=1.$ A similar formula holds for $z_2 =  \left(  \psi^{\mu}_{j} \right)^{-1} (z_1),$  and so we obtain two alternate representations of these branches of $\mathcal{L}_{\lambda}(\phi)$ and $\mathcal{L}_{\mu}(\phi)$. Moreover,  the Lagrange inversion formula implies that $z_2 = \left(  \psi^{\lambda}_{i} \right)^{-1}(z_1)$ and $z_2 = \left(  \psi^{\mu}_{j} \right)^{-1}(z_1)$ have order of contact at $(1,1)$ at least $\K^1_{(1,1)}$. Then Theorem \ref{thm:CO} implies that $\K^{1}_{(1,1)} \le \K^{2}_{(1,1)}$. A symmetric argument gives the other inequality, so we have $\K^{2}_{(1,1)} = \K^1_{(1,1)}$, as needed.
 \end{proof}
 
As the local and global $z_1$- and $z_2$-contact orders are always equal, we can refine our previous definitions of contact order:

\begin{definition} \label{def:CO1v2}  Let $\phi = \frac{\tilde{p}}{p}$ be an RIF on $\mathbb{D}^2$ with a singularity at $(\tau_1, \tau_2)$  on $\mathbb{T}^2$. Define the \emph{contact order of $\phi$ at $(\tau_1,\tau_2)$} to be
\[ \K_{(\tau_1, \tau_2)} := \K^1_{(\tau_1, \tau_2)} = \K^2_{(\tau_1, \tau_2)},\]
where $\K^1_{(\tau_1, \tau_2)} $ and $\K^2_{(\tau_1, \tau_2)} $  are defined in Definition \ref{def:CO1} and shown to be equal in Theorem \ref{thm:ECO}.  Similarly, we can define the \emph{global contact order of $\phi$} to be:
\[K : = K_1 =K_2\]
where $K_1$ and $K_2$ are the global $z_1$- and $z_1$-contact orders from Definition \ref{def:CO1}, which are equal by Theorem \ref{thm:ECO}.  
\end{definition}

One surprising corollary of Theorem \ref{thm:ECO} is that the two partial derivatives of an RIF always possess the same integrability near a singular point:
 
 \begin{corollary} \label{cor:integral} Let $\phi$ satisfy (A1). Then there is an open set $E_0 \subseteq \mathbb{T}^2$ containing $(1,1)$ so that for $1 \le \p < \infty$, and all open 
 sets $E\subseteq E_0$ containing $(1,1)$, we have
\[ \iint_{E} \left | \tfrac{\partial \phi}{\partial z_1}(\zeta_1, \zeta_2) \right|^{\mathfrak{p}} |d\zeta_1| |d\zeta_2| < \infty \  \text{ if and only if } \ \iint_{E} \left | \tfrac{\partial \phi}{\partial z_2}(\zeta_1, \zeta_2) \right|^{\mathfrak{p}} |d\zeta_1| |d\zeta_2| < \infty. \]
\end{corollary}

\begin{proof} The proof follows immediately from Theorem \ref{thm:ECO} paired with Theorem \ref{thm:integral}. \end{proof}

We now have several natural numbers associated with a common zero of $p$ and $\tilde{p}$, namely contact orders of branches and intersection multiplicity. As we already observed by example in the Introduction, the contact order $\K_\mathcal{\tau}$ is in general different from intersection multiplicity $N_{\tau}(p,\tilde{p})$ at a singularity $\tau \in \T^2$. 

We proceed to give a more precise description of the relationship between these quantities, as well as the order of vanishing associated with branches of unimodular level curves.
\begin{lemma}\label{lem:combolemmaOC}
Let $\nu, \mu\in \T$ be distinct, assume $\phi$ satisfies (A1), and suppose $\mathcal{L}_{\mu}(\phi)$ is parametrized by $\psi^{\mu}_1, \ldots, \psi^{\mu}_{L_{\mu}}$, while $\mathcal{L}_{\nu}(\phi)$ is parametrized by 
$\psi^{\nu}_1, \ldots, \psi^{\nu}_{L_{\nu}}$ as in \eqref{eqn:LL}. Then 
\[N_{(1,1)}(p, \tilde{p})=\sum_{i=1}^{L_{\mu}}\sum_{j=1}^{L_{\nu}}\kappa^{\mu, \nu}_{i,j}\]
where $\kappa^{\mu, \nu}_{i,j}$ denotes the order of contact of $\psi^{\mu}_i$ and $\psi^{\nu}_j$ at $(1,1)$ in the sense of Definition \ref{def:branchOC}.
 \end{lemma}
\begin{proof}
As in the proof of Theorem \ref{thm:smooth}, we consider
\[p_{\mu}=\tilde{p}-\mu p \quad \textrm{and}\quad p_{\nu}=\tilde{p}-\nu p\] 
and $q_{\mu}$ and $q_{\nu}$ from \eqref{eqn:Q2}. As we require that $\mathcal{L}_{\mu}(\phi)$ and $\mathcal{L}_{\nu}(\phi)$ be parametrized as in \eqref{eqn:LL}, $q_{\mu}(w_1,0)\neq 0$ and $q_{\nu}(w_1,0)\neq 0$, and so satisfy the conditions of Lemma \ref{lem:weier}. Then they each have a complete Weierstrass factorization, so that
\[q_{\mu}(w_1,w_2)=\prod_{i=1}^{L_{\mu}}(w_1-\Psi^{\mu}_i(w_2)) \quad \textrm{and}\quad q_{\nu}(w_1,w_2)=\prod_{j=1}^{L_{\nu}}(w_1-\Psi^{\nu}_j(w_2))\]
for some convergent power series $\Psi^{\mu}_i$ and $\Psi^{\nu}_j$, as in \eqref{eqn:LL2}. 

Note that $N_{\tau}(r,s)=N_{\tau}(r,s+tr)$ for $r,s, t \in \C[z_1,z_2]$. Using this, and computing intersection multiplicity by switching to the upper half-plane, we obtain
\begin{align*}
N_{(1,1)}(p, \tilde{p})&=N_{(1,1)}(\tilde{p}-{\mu}p, \tilde{p}-{\nu}p)\\
&=N_{(0,0)}(q_{\mu}, q_{\nu})\\
&=\sum_{i=1}^{L_{\mu}}\sum_{j=1}^{L_{\nu}}N_{(0,0)}(w_1-\Psi^{\mu}_i, w_1-\Psi^{\nu}_j).
\end{align*}
Each $N_{(0,0)}(w_1-\Psi^{\mu}_i, w_1-\Psi^{\nu}_j)$ is given by the order of vanishing of the resultant, or in other words, by the order of vanishing of $\Psi^{\mu}_i-\Psi^{\nu}_j$. Since order of contact is invariant under our typical change of variables, the needed statement follows.
\end{proof}
Here is our main result concerning intersection multiplicity and contact order.
\begin{proposition}\label{prop:COvsIM}
Let $\phi$ satisfy (A1) and suppose $\mathcal{V}$ is an open set containing $(1,1)$ such that $\mathcal{Z}_{\tilde{p}}\cap \mathcal{V}$ is described by $z_1=\psi^0_1(z_2)$, $\ldots$, $z_1=\psi^0_{L_0}(z_2)$, as in Theorem \ref{thm:zero}. Then
\[N_{(1,1)}(p,\tilde{p})\leq \sum_{i=1}^{L_0}\sum_{j=1}^{L_0}\min\{\mathcal{K}^1_i, \mathcal{K}^1_j\},\]
where the $\mathcal{K}_i^1$'s are the local contact orders of the branches $\psi^0_i$, $i=1,\ldots, L_0$, at $(1,1)$.
\end{proposition}
\begin{proof}
As in the proof of Theorem \ref{thm:zero} and the beginning of Subsection \ref{subsec:IM}, we switch to the bi-upper half-plane to obtain polynomials $q$ and $\overline{q}$, and functions $\Psi^0_1, \ldots, \Psi^0_{L}$ that generate $\psi^0_1, \ldots, \psi^0_{L_0}$. As was explained in Section \ref{sec:prelim} and \cite[Appendix C]{Kne15}, the desired intersection multiplicity can be computed from $N_{(0,0)}(q, \overline{q})=\sum_{I,J}N_{(0,0)}(q_I, \overline{q}_J)$, where each $q_I$ is an irreducible Weierstrass polynomial of degree $N_I$, and each $N_{(0,0)}(q_I,\overline{q}_J)$ is given by the order of vanishing of 
\[\prod_{i=1}^{N_I} \prod_{\ell=1}^{N_J} \left(  \Psi^0_I\left ( \zeta^i t^{\frac{1}{N_I}}\right) -  \bar{\Psi}_J^0\left( \eta^\ell  t^{\frac{1}{N_J}} \right) \right),\]
where $\zeta$ and $\eta$ are primitive roots of unity. Moreover, recall that $L_0 = N_1 + \dots + N_{L}$. Hence it suffices to establish that, for any fixed pair of indices $(I,J)$ 
and choice of $i$ and $\ell$, the vanishing order of $\Psi^0_I\left ( \zeta^i t^{\frac{1}{N_I}}\right) -  \bar{\Psi}_J^0\left( \eta^\ell  t^{\frac{1}{N_J}} \right)$ is at most $\min\{\mathcal{K}^1_I, \mathcal{K}^1_J\}$.

Without loss of generality, suppose $\mathcal{K}^1_I\leq \mathcal{K}^1_J$. As in Section \ref{sec:prelim}, we have
\[\Psi^0_I(t)=\sum_{k=1}^{2M-1}b_k^It^{kN_I}+b_{2M}^It^{2MN_I}+\sum_{k=2MN_I+1}^{\infty}a_{k}^It^k,\]
where $N_I$ is a positive integer, $b_1^I, \ldots, b_{2M-1}^I$ are real, and $\mathcal{K}^1_I=2M$. From \cite[Appendix C]{Kne15} we moreover know that $\Im(b_{2M})>0$. A similar expansion, with coefficients denoted by $b_k^J$, holds for $\Psi^0_J$. 

If, for some $k\leq 2M-1$, we have $b_{k}^I-b_k^J\neq 0$, it follows that the order of vanishing of $\Psi^0_I\left ( \zeta^i t^{\frac{1}{N_I}}\right) -  \bar{\Psi}_J^0\left( \eta^\ell  t^{\frac{1}{N_J}} \right)$ is strictly smaller than $\mathcal{K}^1_I$, so that the desired inequality holds. Suppose then that the real-coefficient terms in $\Psi^0_I\left ( \zeta^i t^{\frac{1}{N_I}}\right) -  \bar{\Psi}_J^0\left( \eta^\ell  t^{\frac{1}{N_J}} \right)$ cancel; we need to argue that we cannot have additional cancellation in front of $t^{\mathcal{K}^1_IN_I}$ and thus higher order of vanishing. But this now follows from the definition of $\bar{\Psi}^0_J$ and the fact that $\Im(b_{2M}^I)$ is positive and $\Im(b_{2M}^J)$ is non-negative: either $b_{2M}^J$ is real (if $\mathcal{K}^1_J>\mathcal{K}^1_I$), or else $\Im(b_{2M}^J)>0$ (if $\mathcal{K}^1_I=\mathcal{K}^1_J$). 

The proof is now complete.
\end{proof}


\section{Fine contact order vs fine order of contact}\label{sec:COfine}

In this section, we further examine the relationship between the contact order and order of contact of an RIF at a singular point. In Section \ref{sec:CO}, we examined these quantities at a fixed singularity. Now, we consider these quantities at the level of branches or curves. Specifically, 
we will connect the contact order associated with a specific branch of $\mathcal{Z}_{\tilde{p}}$ with the order of contact between two particular branches of the unimodular curves $\mathcal{L}_{\mu}(\phi)$ and $\mathcal{L}_{\lambda}(\phi).$ 

Assume $\phi$ satisfies (A1). To make sense of the main result, recall that near $(1,1)$, the zero set $\mathcal{Z}_{\tilde{p}}$ has $L_0$ branches 
\[  z_1 = \psi_1^0(z_2), \ \dots, \ z_1 = \psi_{L_0}^0(z_2),\]
as given in \eqref{eqn:zero1}. Similarly, for $\mu \in \mathbb{T}$, the unimodular level curve $\mathcal{L}_{\mu}(\phi)$ is comprised of $L_{\mu}$ smooth curves
\[ z_1 = \psi^{\mu}_{1}(z_2), \ \dots, z_1 =  \psi^{\mu}_{L_\mu}(z_2),\]
as given by \eqref{eqn:LL}, and possibly a vertical component. Then here is the precise result:

\begin{theorem} \label{prop:bijection1} Let $\phi$ satisfy (A1).  Then for almost every pair $\lambda, \mu \in \mathbb{T}$, we have $L_{\lambda}, L_{\mu} \ge L_0$. Furthermore, after a reordering of the components of $\mathcal{L}_{\lambda}(\phi)$ and $\mathcal{L}_{\mu}(\phi)$ near $(1,1)$, the $z_1$-contact order of $z_1 = \psi^0_\ell(z_2)$ at $(1,1)$ is at most the order of contact between $z_1 = \psi^{\mu}_{\ell}(z_2)$ and  $z_1 = \psi^{\lambda}_{\ell}(z_2)$ at $(1,1)$ for $1 \le \ell \le L_0.$
\end{theorem}

\begin{proof} The proof is a more technical version of the proof of Theorem \ref{thm:CO1}.  As in that proof, fix $\zeta \in \mathbb{T}\setminus \{1\}$ near $1$. Then $\phi_{\zeta}(z_1):= \phi(z_1, \zeta)$ is a finite Blaschke product and $\phi(  \psi^0_{\ell}(\zeta), \zeta)=0$ for $1\le \ell \le L_{0}$. For $\zeta$ close enough to $1$, the $\psi^0_{\ell}(\zeta)$ are distinct, and so, the product $\prod_{\ell=1}^{L_0} b_{\psi^0_{\ell}(\zeta)}$ divides $\phi_{\zeta}.$

Fix $\epsilon>0$ and let $(\zeta_n) \subseteq \mathbb{T}$ be a sequence converging to $1$ with each $\zeta_n\neq 1$. For each $n \in \mathbb{N}$ and $\ell$ with $1 \le \ell \le L_0,$ let $A^n_{\ell,\epsilon}: = A_{\epsilon, \psi^0_{\ell}(\zeta_n)}$ denote the arc from Lemma \ref{lem:move}. Note that the sets $A^n_{\ell,\epsilon}$, $\ell=1,\ldots, L_0$ need not be disjoint. By initially reordering the components of $\mathcal{Z}_{\tilde{p}}$ near $(1,1)$ and then passing to a subsequence, we can assume 
\begin{equation}
 \left | A^n_{1,\epsilon} \right | \le \dots \le \left | A^n_{L_0,\epsilon}\right|,
\label{eq:sortbymeas}
\end{equation}
for all $n \in \mathbb{N}.$  Define $\mathcal{C}^n_{\ell,\epsilon} = \cup_{i=1}^{\ell} A^n_{i,\epsilon}$ and let $\mathcal{D}^n_{\ell,\epsilon}$ denote the connected component of $\mathcal{C}^n_{\ell,\epsilon}$ that contains $A^n_{\ell,\epsilon}.$ Moreover, let $N_{\ell,\epsilon}$ denote the number of $A^n_{i,\epsilon}$ contained in $\mathcal{D}^n_{\ell,\epsilon}.$ While technically, $N_{\ell,\epsilon}$ depends on $n$, by passing to another subsequence, we can assume each $N_{\ell,\epsilon}$ is independent of $n$. Moreover, $\left | \mathcal{D}^n_{\ell,\epsilon} \right| \le \ell \cdot \left| A^n_{\ell,\epsilon}\right|$ in view of \eqref{eq:sortbymeas}. Now define the image set
\[ I^n_{\ell,\epsilon}:=\left \{  \phi(\tau, \zeta_n) : \tau \in  \mathcal{D}^n_{\ell,\epsilon} \right \}.\]
As $\phi_{\zeta_n}$ is continuous on $\mathbb{T}$, we know $I^n_{\ell,\epsilon}$ is an arc that winds around $\mathbb{T}$ and by Lemma \ref{lem:move},
\[ \left | I^n_{\ell,\epsilon} \right|_{\mathcal{W}} =
 \mu_{\phi_{\zeta_n}}\left( \mathcal{D}_{\ell,\epsilon} \right)\ge 
 \sum_{i=1}^{\ell}\mu_{b_{\psi^0_i(\zeta_n)}}(D_{\ell, \epsilon})
   \ge N_{\ell,\epsilon}\cdot\left( 2\pi  -\epsilon \right).\]
Then each $I^n_{\ell,\epsilon}$ yields an arc $T_{\ell,\epsilon}^n \subseteq \mathbb{T}$ of distinct points with $|T_{\ell,\epsilon}^n| \ge 2\pi  - N_{\ell, \epsilon} \cdot \epsilon$ so that for each $\lambda \in T_{\ell, \epsilon}^n$, there are $N_{\ell,\epsilon}$ occurrences of $\lambda$ in $I^n_{\ell,\epsilon}.$ Using the same arguments as in the proof of Theorem \ref{thm:CO1}, we can pass to a subsequence and obtain for each $\ell$ an arc $B_{\ell, \epsilon}$ so that the length $\left | B_{\ell, \epsilon}\right | =  2 \pi - 2 N_{\ell,\epsilon}\cdot \epsilon$ 
and for all $n$ sufficiently large, $B_{\ell, \epsilon} \subseteq T_{\ell,\epsilon}^n$.
Let $B_{\epsilon} = \cap B_{\ell, \epsilon}.$ Then $B_{\epsilon}$ is a union of arcs in $\mathbb{T}$ with 
\[ \left | B_{\epsilon}\right | \ge  2 \pi - 2 (N_{1,\epsilon} + \dots +N_{L_0, \epsilon}) \epsilon.\] 
Indeed, $B_{\epsilon}$ can be obtained from $\mathbb{T}$ by omitting at most $L_0$ intervals of length at most $2 (N_{1, \epsilon} + \dots +N_{L_0, \epsilon}) \epsilon.$

Then for each $\lambda \in B_{\epsilon}$, $n$ sufficiently large, and $\ell$ with $1 \le \ell \le L_0,$ this construction gives $N_{\ell, \epsilon}$ distinct elements from $\mathcal{L}_{\lambda}(\phi)$ in each $\mathcal{D}^n_{\ell}$. To be specific, the process is as follows:
\begin{itemize}
\item[1.]  As $\lambda \in T_{1,\epsilon}^n,$ there is a $\tau_1^1 \in \mathcal{D}^n_{1, \epsilon}$ with $\phi(\tau_1^1,\zeta_n) =\lambda$. As long as $n$ is sufficiently large, we can choose $\tau_1^1 = \psi^{\lambda}_{j_1}(\zeta_n)$ for some $j_1$ with $1 \le j_1 \le L_{\lambda}.$ 
\item[2.] As $\lambda \in T_{2, \epsilon}^n,$ there is a $\tau_1^2 \in \mathcal{D}^n_{2, \epsilon}$ with $\phi(\tau_1^2,\zeta_n) = \lambda$. We can further choose $\tau_1^2 \ne \tau_1^1$. Indeed, if $\tau_1^1 \in \mathcal{D}^n_{2, \epsilon}$, then $A^n_{1,\epsilon} \subseteq \mathcal{D}^n_{2,\epsilon}$ and so by construction, there are two occurrences of $\lambda \in I_{2,\epsilon}^n.$ Thus, we can choose $\tau_2^1 \ne \tau_1^1.$ Then as long as $n$ is sufficiently large, we can choose $\tau_1^2 = \psi^{\lambda}_{j_2}(\zeta_n)$ for some $j_2$ with $1 \le j_2 \le L_{\lambda}$ and $j_1 \ne j_2.$ 
\item[3.] We can continue in this matter. For each $\ell$ with $1\le \ell \le L_0$, we can identify a point $\psi^{\lambda}_{j_\ell}(\zeta_n) \in \mathcal{D}^n_{\ell,\epsilon}$, where $j_{\ell} \ne j_1, \dots, j_{\ell-1}.$
\end{itemize}
 Now assume  $\lambda, \mu \in B_{\epsilon}$. By reordering the components of $\mathcal{L}_{\lambda}(\phi)$ and $\mathcal{L}_{\mu}(\phi)$ and passing to a subsequence, we can further assume that our arguments give $ \psi^{\lambda}_{\ell}(\zeta_n), \psi^{\mu}_{\ell}(\zeta) \in \mathcal{D}^n_{\ell, \epsilon}$ for each $\ell$ with $1 \le \ell \le L_0$ and all $n$ sufficiently large.  This immediately implies that $L_{\lambda}, \mathcal{L}_{\mu} \ge L_0.$
  
Then we have $ \psi^{\lambda}_{\ell}(\zeta_n), \psi^{\mu}_{\ell}(\zeta_n) \in \mathcal{D}^n_{\ell,\epsilon}$, for $n$ sufficiently large.   Fix $\ell$ with $1 \le \ell \le L_0$ and let $\K^1_{\ell}$ denote the $z_1$-contact order of $z_1 = \psi_{\ell}^{0}(z_2)$ at $(1,1)$.  Then
\[
\left |  \psi^{\lambda}_{\ell}(\zeta_n)- \psi^{\mu}_{\ell}(\zeta_n) \right | \le  \left| \mathcal{D}^n_{\ell, \epsilon} \right| \\
 \le \ell \cdot \left | A^n_{\ell, \epsilon} \right|
 \lesssim c_{\epsilon} \left( 1 - | \psi_{\ell}^{0}(\zeta_n) |\right) \approx c_{\epsilon} \left | 1 - \zeta_n \right |^{\K^1_{\ell}},
\]
for large enough $n$. By the smoothness of the branches, we know that the $z_1$-order of contact between $z_1 =  \psi^{\lambda}_{\ell}(\zeta_n)$ and $z_1= \psi^{\mu}_{\ell}(\zeta_n)$ at $(1,1)$ is at least
$\K^1_{\ell}.$ 

Finally, we claim that for almost every pair $\mu, \lambda \in \mathbb{T}$, there is an $\epsilon>0$ so that $\mu, \lambda \in B_{\epsilon}.$ 
In particular, proceeding towards a contradiction, let $L=L_0+1$ and assume there are pairs $\lambda_1, \mu_1, \dots, \lambda_L, \mu_L$, such that 
each pair $\lambda_i, \mu_i$ is not in a common $B_{\epsilon}$ and every $\lambda_i \ne \lambda_j$ and $\mu_i \ne \mu_j$. Fix a sequence $(\epsilon_m)$ of positive numbers converging to $0$. By passing to a subsequence and switching any $\lambda_i$ with $\mu_i$  if necessary, we can assume that each
$B_{\epsilon_m}$ omits every $\mu_1, \dots, \mu_L.$ Recall that each $B_{\epsilon_m}$ can be obtained from $\mathbb{T}$ by omitting at most $L_0$ intervals of length at most $2 (N_{1, \epsilon_m} + \dots +N_{L_0, \epsilon_m}) \epsilon_m.$ Thus, as every $\mu_i \ne \mu_j$,  if we choose $\epsilon_m>0$ sufficiently small, $B_{\epsilon_m}$ can omit at most $L_0$ of $\mu_1, \dots, \mu_L$, a contradiction.

Thus, for almost every pair $\mu, \lambda \in \mathbb{T}$, there is an $\epsilon>0$ so that $\mu, \lambda \in B_{\epsilon}.$ Then our previous arguments imply that, up to reordering, the $z_1$-order of contact between $z_1 =  \psi^{\lambda}_{\ell}(z_2)$ and $z_1= \psi^{\mu}_{\ell}(z_2)$ at $(1,1)$ is at least
the $z_1$-contact order of $z_1 = \psi^0_\ell(z_2)$ at $(1,1)$ for $1 \le \ell \le L_0.$ 
\end{proof}

We conjecture that the following refined result is also true:

\begin{conjecture}  Let $\phi$ satisfy (A1). Then for almost every pair $\lambda, \mu \in \mathbb{T}$, we have $L_{\lambda} = L_0=L_{\mu}$. Furthermore, after a reordering of the components of $\mathcal{L}_{\lambda}(\phi)$ and $\mathcal{L}_{\mu}(\phi)$ near $(1,1)$, the contact order of $z_1 = \psi^0_\ell(z_2)$ at $(1,1)$ will equal the order of contact between $z_1 = \psi^{\mu}_{\ell}(z_2)$ and  $z_1 = \psi^{\lambda}_{\ell}(z_2)$ at $(1,1)$ for $1 \le \ell \le L_0.$
\end{conjecture}

\section{Constructions of rational inner functions}\label{sec:construct}
We now present several methods of constructing RIFs with desired level set behavior.

\subsection{One Prescribed Level Set}

For our initial construction, we consider functions similar to those studied in \cite{Kne10TAMS} and use them to construct RIFs with one prescribed unimodular level set.
A polynomial $r\in \C[z_2,z_2]$ is called $\T^2$-symmetric if $\tilde{r}=\lambda r$ for some unimodular constant $\lambda$ (cf. \cite[p.5638]{Kne10TAMS}).
\begin{theorem} \label{thm:LS} Let $r\in \mathbb{C}[z_1, z_2]$ be non-constant and essentially $\T^2$-symmetric, say $\tilde{r}=\lambda r$, with no zeros on $\mathbb{D}^2$. Then there is an RIF $\phi$ on $\mathbb{D}^2$ such that $\mathcal{L}_{\lambda}(\phi) = \mathcal{Z}_r.$
\end{theorem}

\begin{proof} Fix such an $r$ and define the polynomial
\[ \tilde{p}(z_1,z_2) = z_1 \tfrac{\partial r}{\partial z_1}(z_1,z_2)  + z_2 \tfrac{ \partial r}{\partial z_2}(z_1,z_2) .\]
Define $p = \tilde{\tilde{p}}$. Then we claim $\phi:= -\frac{\tilde{p}}{p}$
 is an RIF on $\mathbb{D}^2$ and $\mathcal{L}_{\lambda}(\phi) = \mathcal{Z}_r.$ By construction, it is immediate that $|\phi|=1$ on $\mathbb{T}^2$ and $\phi$ is rational.  To see that $\phi$ has no singularities in $\mathbb{D}^2$, fix $t$ with $0<t<1$ and set $r_t(z_1,z_2) :=r(z_1t, z_2t)$. Then $r$ does not vanish on $\overline{\mathbb{D}^2}$ and so each $f_t : = \frac{\tilde{r}_t}{r_t}$ is a non-constant RIF continuous on $\overline{\mathbb{D}^2}$. This means that for each $t$, we can also define another RIF on $\mathbb{D}^2$ by
\[\phi_t(z_1,z_2) := -\frac{ f_t(z_1,z_2)  - f_t(0,0)}{1-\overline{f_t(0,0)} f_t(z_1,z_2)}.\]
 A simple application of L'Hopital's Rule implies that for each $(z_1,z_2) \in \mathbb{D}^2$, 
\[ \phi(z_1,z_2)  = \lim_{t\nearrow 1} \phi_t(z_1,z_2).\]
This implies $\phi$ cannot have any singularities in $\mathbb{D}^2$. Thus, $\phi$ is an RIF.  Lastly if $\deg p = (m,n)$, then a simple computation gives  $\lambda p + \tilde{p} = r (m+n)$. Thus $\mathcal{L}_{\lambda}(\phi) = \mathcal{Z}_r,$ as needed. 
\end{proof}

We call a zero variety $\mathcal{Z}_{r}$ associated to an essentially $\T^2$-symmetric polynomial $r\in \mathbb{C}[z_1,z_2]$ that does not vanish in the bidisk 
a \emph{codistinguished variety}. Theorem \ref{thm:smooth} now immediately yields:
\begin{corollary}\label{cor:codist}
Codistinguished varieties intersect $\T^2$ along smooth curves.
\end{corollary}
This observation can be used to simplify the proof of \cite[Theorem 5.2]{BKKLSS15}.

\subsection{Gluing two level sets}
Given an RIF, we can also construct a new RIF with a unimodular level set obtained by ``gluing'' together two unimodular level sets from the original RIF. Specifically:
\begin{corollary}\label{cor:glue} 
Let $\phi = \frac{\tilde{p}}{p}$ be a non-constant RIF on $\mathbb{D}^2$. Then there is an RIF $\Phi$ on $\mathbb{D}^2$ such that $\mathcal{L}_1(\Phi) = \mathcal{L}_i(\phi) \cup \mathcal{L}_{-i}(\phi).
$
\end{corollary}

\begin{proof}
 Let $\phi = \frac{\tilde{p}}{p}$ be an RIF on $\mathbb{D}^2$ and define
\begin{equation} \label{eqn:r} r(z_1, z_2):=(p(z_1,z_2))^2+(\tilde{p}(z_1,z_2))^2.\end{equation}
Then $r$ satisfies the conditions of Theorem \ref{thm:LS}. Thus, if we set
\[\tilde{P}(z_1,z_2)=z_1\tfrac{\partial r}{\partial z_1}(z_1,z_2)+z_2\tfrac{\partial r}{\partial z_2}(z_1,z_2),\]
and reflect to obtain $P$, the RIF $\Phi  = \frac{\tilde{P}}{P}$ satisfies $\mathcal{L}_1(\Phi) = \mathcal{Z}_r$. Finally, the identity $(\tilde{p}+ip)(\tilde{p}-ip)=p^2+\tilde{p}^2$ shows that $\mathcal{Z}_r= \mathcal{L}_i(\phi) \cup \mathcal{L}_{-i}(\phi)$, as needed.
\end{proof}

\subsection{Interlacing Constructions}

 In Theorem \ref{thm:LS}, we showed that if $r = \tilde{r}$ is non-constant and $\mathcal{Z}_r \cap \mathbb{D}^2 = \emptyset,$ then there is an RIF $\phi$ with $\mathcal{L}_1(\phi) = \mathcal{Z}_r.$ In this section, we obtain necessary and sufficient conditions to specify two unimodular level curves of an RIF. In particular, we will answer the following question:
\[ \text{Given $r, q \in \mathbb{C}[z_1, z_2]$, when is there an RIF $\phi$ with $\mathcal{L}_1(\phi) = \mathcal{Z}_{r}$ and $\mathcal{L}_{-1}(\phi) = \mathcal{Z}_{q}$?}\]

To simplify the problem, we will switch to the bi-upper half plane $\Pi^2$. In particular, recall the conformal map $\beta: \Pi \rightarrow \mathbb{D}$ from \eqref{eqn:beta} that satisfies $\beta(0)=1$ and $\beta(\infty)=-1.$ The needed formulas are $\beta(w)= \frac{1+iw}{1-iw}$ and $\beta^{-1}(z) = i \frac{1-z}{1+z}.$  Further recall that $\Phi$ is a rational inner Pick function (RIPF) on $\Pi^2$ if $\Phi: \Pi^2 \rightarrow \Pi$ is a rational function with no poles on $\Pi^2$ satisfying $\text{Im}(\Phi(x))=0$ for a.e.~$x=(x_1, x_2) \in \mathbb{R}^2$. 

Given $r, q\in \mathbb{C}[z_1, z_2]$ with $\deg r = (m,n) = \deg q$, define the following polynomials:
\begin{equation} \label{eqn:RQ} R(w) := (1-iw_1)^m(1-iw_2)^n r( \beta(w)) \ \text{ and } \
Q(w) := (1-iw_1)^m(1-iw_2)^n q( \beta(w)). 
\end{equation}
Here, $r( \beta(w))$ is shorthand for $r(\beta(w_1), \beta(w_2))$, and this notation will be used throughout the following proof.  
Then we have the following lemma:

\begin{lemma} \label{lem:poly} Let $r, q\in \mathbb{C}[z_1, z_2]$ with no common factors and $\deg r = (m,n) = \deg q$. Then there is an RIF $\phi$  on $\mathbb{D}^2$ with $\deg \phi = (m,n)$ so that $\mathcal{L}_1(\phi)=\mathcal{Z}_r$ and $\mathcal{L}_{-1}(\phi) = \mathcal{Z}_q$ if and only if there is a nonzero constant $c$ such that $\Phi:=c \frac{R}{Q}$ is a RIPF on $\Pi^2$.
\end{lemma}

\begin{proof} 

($\Rightarrow$) Assume there exists a rational inner $\phi$ with  $\deg \phi = (m,n)$ so that $\mathcal{L}_1(\phi)=\mathcal{Z}_r$ and $\mathcal{L}_{-1}(\phi) = \mathcal{Z}_q$. We can write $\phi(z) = b(z) \frac{\tilde{p}}{p}(z)$ for $b(z)$ a monomial and $p, \tilde{p}$ with no common factors. This implies  that there are nonzero constants $c_1, c_2$ such that 
\[ r(z) = c_1\left(p(z)-b(z)\tilde{p}(z) \right) \ \text{ and } \ q(z) = c_2\left(p(z)+b(z)\tilde{p}(z) \right).\]
Define the rational inner Pick function $\Phi :=\beta^{-1} \circ \phi \circ \beta$. Then 
\[ \Phi(w) = i \left( \frac{1-\phi}{1+\phi}\right)\big( \beta(w)\big) = i \frac{c_2}{c_1} \frac{r}{q}\big( \beta(w)\big) = i \frac{c_2}{c_1} \frac{R(w)}{Q(w)},\]
as needed.\\

($\Leftarrow$) Similarly, assume that there is a nonzero constant $c$ so that $\Phi := c \frac{R}{Q}$ is a RIPF on $\Pi^2.$ Setting $\phi = \beta \circ \Phi \circ \beta^{-1}$ and working through the definitions gives
\[ \phi(z)  = \left( \frac{Q + i cR}{Q-icR} \right) \big( \beta^{-1}(z) \big) = \frac{q(z) +icr(z)}{ q(z) - icr(z)}.\]
Then  $\mathcal{L}_1(\phi)=\mathcal{Z}_r$ and $\mathcal{L}_{-1}(\phi) = \mathcal{Z}_q.$ Since $q$ and $r$ have no common factors and satisfy $\deg r = (m,n) = \deg q$, we can conclude $\deg \phi =(m,n)$ as well.
\end{proof}

By Lemma \ref{lem:poly}, we only need to characterize when a rational function $\Phi := c \frac{R}{Q}$ is a RIPF on $\Pi^2$. First, we consider the one-variable situation. The following result is likely well known but we include its proof for completeness.

\begin{lemma} \label{lem:RIFP1} Let $R, Q \in \mathbb{C}[z]$ be nontrivial with no common zeros and  let $C$ be the ratio of their leading coefficients. Then $\Phi := \frac{R}{Q}$ 
is a rational inner Pick function on $\Pi$ if and only if $R$ and $Q$ have only real zeros, say $a_1$,\dots, $a_m$ and $b_1, \dots, b_n$ respectively satisfying
$n-1 \le m \le n+1$ so that if the zeros were listed in increasing order, then:
\begin{itemize}
\item[(i)] If $m=n-1$, then $C<0$ and $b_1 < a_1 < b_2 < \dots <a_{n-1} <b_n.$
\item[(ii)] If $m=n$, then either
\begin{itemize}
\item[(a)] $C<0$ and $a_1 < b_1 < \dots <a_n < b_n$, or 
\item[(b)] $C>0$ and $b_1 < a_1 < \dots <b_n < a_n$. 
\end{itemize}
\item[(iii)] If $m=n+1$, then $C>0$ and $a_1 <b_1 <a_2<\dots < b_n <a_{n+1}.$
\end{itemize}
\end{lemma}

\begin{proof} Recall \cite[p.19]{DonBook} that $\Phi= \frac{R}{Q}$ is a rational inner Pick function if and only if 
\begin{equation} \label{eqn:RIFP} \Phi(w) = \delta w + \gamma + \sum_{i=1}^n \frac{r_i}{w-b_i},\end{equation}
for some $\delta \ge 0$, $\gamma \in \mathbb{R}$ and each $r_i \le 0$. As part of this formula, the poles $b_1, \dots, b_n$
are real and distinct. Observe that if $\Phi$ satisfies \eqref{eqn:RIFP}, then the number of zeros $m \le n+1.$ We will find necessary and sufficient conditions for $\Phi$ to satisfy \eqref{eqn:RIFP}. 

($\Rightarrow$) Assume $\Phi$ satisfies \eqref{eqn:RIFP}. By assumption we can write 
\begin{equation} \label{eqn:RIFP2} \Phi(w) = C \frac{ \prod_{j=1}^m (w - a_j)}{ \prod_{i=1}^n (w - b_i)}.\end{equation}
Looking at \eqref{eqn:RIFP}, we can conclude that the coefficients of the numerator  $C\prod_{j=1}^m (w - a_j)$ must be real. This means that its zeros must be real or occur in complex conjugate pairs. Since none of the zeros can occur in $\Pi$, all of the zeros must be real.

Now observe that each $r_k = \left( \Phi(w)(w-b_k)\right)_{w=b_k}$ and so 
\begin{equation} \label{eqn:sgn} \text{sgn}(r_k) = \text{sgn}(C) \prod_{j=1}^m \text{sgn}(b_k-a_j) \prod_{i\ne k}(b_k-b_i).\end{equation}
To ensure the $r_k$ all have the same sign (negative), we need an odd number of zeros 
between each two consecutive poles. This implies that $\Phi$ has at least $n-1$ zeros. Thus, 
we can conclude that $n-1 \le m \le n+1.$ Consider each case:\\

\noindent \textbf{Case 1:} Assume $m=n-1$. Then by our previous observation, there is one zero between each
pair of consecutive poles. This implies that $b_1 < a_1 < b_2 < \dots <a_{n-1} <b_n.$ Then \eqref{eqn:sgn} becomes
\[  \text{sgn}(r_k) = \text{sgn}(C) (-1)^{n-k} (-1)^{n-k} =  \text{sgn}(C) \le 0,\]
 and so $C<0$.\\

\noindent \textbf{Case 2:} Assume $m=n$.  If $m=n$ it is not possible to have three zeros between any two consecutive poles. Thus, there must be exactly one zero
between each pair of consecutive poles, which implies the zero and pole configuration must be either
\[ a_1 < b_1 <a_2 < \dots <a_n <b_n \  \text{ or }  \ b_1 < a_1 <b_2 < \dots < b_n < a_n.\]
If the first configuration occurs, then each $r_k =\text{sgn(C)} (-1)^{n-k}(-1)^{n-k}=\text{sgn(C)}$ and so, we must have $C<0.$ Similarly, if the second configuration occurs, then each $r_k =\text{sgn(C)} (-1)^{n-k}(-1)^{n-k+1}= -\text{sgn(C)},$ and so we must have $C>0$.\\

\noindent \textbf{Case 3:} Assume $m=n+1$. Observe that in this case, $\delta = C$. Since $\delta \ge 0$, we automatically get $C>0$. 
Let $M$ denote the number of zeros larger than $b_1$. Because we need an odd number of zeros between consecutive poles, we know $ n-1 \le M \le n+1.$
 Then
\[ \text{sgn}(r_1) = \prod_{j=1}^m \text{sgn}(b_1-a_j) \prod_{i\ne 1}(b_1-b_i) =  (-1)^M (-1)^{n-1}.\]
As $r_1 \le 0$, we must have $M=n$. This immediately implies that the zeros must satisfy
\[ a_1 < b_1 < \dots < a_n < b_n < a_{n+1}.\]
Thus, a rational inner Pick function must satisfy the given conditions.\\

($\Leftarrow$) Assume $\Phi = \frac{R}{Q}$, where $R$ and $Q$ have only real zeros, say $a_1$,\dots, $a_m$ and $b_1, \dots, b_n$ respectively, satisfying
$n-1 \le m \le n+1$ and either $(i)$, $(ii)$, or $(iii)$. We must show that $\Phi$ satisfies \eqref{eqn:RIFP}. Using its partial fraction decomposition, we can write $\Phi$ 
in the form  \eqref{eqn:RIFP}; thus, we just need to verify that $\delta \ge 0$, $\gamma \in \mathbb{R}$, and each $r_i \le 0$.  First, observe that in cases $(i)$ and $(ii)$, $\delta$ is automatically zero, since $\deg R \le \deg Q.$ Similarly, in case $(iii)$, $\delta = C >0$. Similarly, in each case, \eqref{eqn:sgn} paired with the appropriate zero configuration and $\text{sgn}(C)$ implies that each $r_i \le 0.$ Finally, given the other coefficients, if $\gamma \not \in \mathbb{R}$, then for $x\in \R$ with $x \ne b_j$, we have $\Phi(x) \not \in \mathbb{R}.$ But formula \eqref{eqn:RIFP2} implies that such $\Phi(x) \in \mathbb{R}$, a contradiction. 
\end{proof}

We can use this result to classify when a ratio of polynomials yields a rational inner Pick function in two variables. First for any two-variable $R \in \mathbb{C}[w_1, w_2]$, we will define one variable polynomials as follows. Fix $x = (x_1, x_2) \in \mathbb{R}^2$ and $y = (y_1, y_2) \in \mathbb{R}^2_+$. Then $R_{x,y}$ denotes the one-variable polynomial
\begin{equation}
R_{x,y}(w) : = P(x_1 + y_1 w, x_2 + y_2 w).
\label{eqn:sliceR}
\end{equation}
Given these slices, we have the following result:

\begin{theorem} \label{thm:interlacing} Let $R, Q \in \mathbb{C}[w_1, w_2]$ be nontrivial with no common factors. For each $x \in \mathbb{R}^2$ and $y \in \mathbb{R}^2_+$, let $R_{x,y}$ and $Q_{x,y}$ denote the associated one-variable polynomials as in \eqref{eqn:sliceR}, and let $C_{x,y}$ denote the ratio of their leading coefficients. Then $\Phi = \frac{R}{Q}$ is a two-variable rational inner Pick function if and only if, after canceling common factors, $R_{x,y}$, $Q_{x,y}$, and $C_{x,y}$ satisfy one of $(i)$, $(ii)$, or $(iii)$ from Lemma \ref{lem:RIFP1} for all $x \in \mathbb{R}^2$ and $y \in \mathbb{R}^2_+$. 
\end{theorem}

\begin{proof} ($\Rightarrow$) Assume $\Phi$ is a two-variable rational inner Pick function. Fix any $x \in \mathbb{R}^2$ and $y \in \mathbb{R}^2_+$. Then after canceling common real zeros, $\Phi_{x,y} :=\frac{R_{x,y}}{Q_{x,y}}$ is rational, maps $\Pi \rightarrow \Pi$ and except at the zeros of $Q_{x,y}$, maps $\mathbb{R}$ to $\mathbb{R}$. Thus, $\Phi_{x,y}$ is a one variable rational inner function and so after canceling common real zeros,  Lemma \ref{lem:RIFP1} implies that $R_{x,y}$, $Q_{x,y}$, and $C_{x,y}$ satisfy one of $(i)$, $(ii)$, or $(iii)$.\\

\noindent ($\Leftarrow$) Clearly $\Phi = \frac{R}{Q}$ is rational. Fix any $(w_1, w_2) \in \Pi^2$. Then there is some $x \in \mathbb{R}^2$, $y\in \mathbb{R}^2_+$, and $w \in \Pi$ so that 
\[ (w_1, w_2) = (x_1 + y_1 w, x_2 + y_2 w).\]
By assumption, after canceling common real factors, $  \frac{P_{x,y}}{Q_{x,y}}$ is a one-variable rational inner function. This implies that $Q(w_1,w_2)$ is non-zero and moreover, 
\[ \Phi(w_1, w_2)  = \frac{R(w_1, w_2)}{Q(w_1,w_2)} = \frac{R_{x,y}(w)}{Q_{x,y}(w)} \in \Pi,\]
as needed. Thus, $\Phi$ is analytic and maps $\Pi^2$ into $\Pi$. Now fix any $x \in \mathbb{R}^2$ that is not a zero of $Q$. Then, for any $y\in \mathbb{R}_+^2$, 
\[ \Phi(x_1, x_2) = \frac{ R_{x,y}(0)}{Q_{x,y}(0)} \in \mathbb{R},\]
by assumption. This implies that $\Phi$ sends $\mathbb{R}^2$ to $\mathbb{R}$ almost everywhere and so, is a rational inner Pick function of two variables.
\end{proof}

Returning to the original question, let $r, q\in \mathbb{C}[z_1, z_2]$ with no common factors and $\deg r = (m,n) = \deg q$ and define $R$ and $Q$ as in \eqref{eqn:RQ}. Then by Lemma \ref{lem:poly}, there is an RIF $\phi$  on $\mathbb{D}^2$ with $\deg \phi = (m,n)$ so that $\mathcal{L}_1(\phi)=\mathcal{Z}_r$ and $\mathcal{L}_{-1}(\phi) = \mathcal{Z}_q$ if and only if there is a nonzero constant $c$ such that $\Phi:=c \frac{R}{Q}$ satisfies the conditions of Theorem \ref{thm:interlacing}.

Intuitively speaking, Theorem \ref{thm:interlacing} asserts that if along every conformal line we have an interlacing of zeros, then there is an RIF having the desired level curves.

\section{A zoo of rational inner functions}\label{sec:zoo}
We further illustrate the findings in this paper by examining several examples in detail.

\subsection{Contact order and intersection multiplicity are different}
We give a minimal example showing that contact order and intersection multiplicity are different in general. This example is obtained by applying the embedding construction described in Theorem \ref{thm:LS}.

Consider the polynomial $r(z_1,z_2)=(1-z_1)(1-z_2)(1-z_1z_2)$ and set 
\[\tilde{p}(z_1,z_2)=z_1\frac{\partial r}{\partial z_1}(z_1,z_2)+z_2\frac{\partial r}{\partial z_2}(z_1,z_2).\]
Forming $p$ from $\tilde{p}$ by reflecting, and setting $\phi=-\tilde{p}/p$, we obtain the RIF 
\begin{equation}
\phi(z_1,z_2)=-\frac{z_2 + z_1 - 3 z_2^2 z_1 - 3 z_2 z_1^2 + 4 z_1^2 z_2^2}{4 - 3 z_1 -3 z_2 +z_2^2 z_1+
 z_2 z_1^2}
 \label{minimalCOex}
 \end{equation}
which has bidegree $(2,2)$, and a singularity at $(1,1)$ with non-tangential value $\phi(1,1)=-1$. 

A careful analysis shows that $p$ and $\tilde{p}$ have a common zero at $(1,1)$ and additional common zeros at $(0, \infty)$ and $(\infty, 0)$. We now compute intersection multiplicities as in B\'ezout's theorem, using that we only have one singularity on $\T^2$. This yields
\[8=N(p,\tilde{p})=N_{\T^2}(p, \tilde{p})+N_{(0,\infty)}(p, \tilde{p})+N_{(\infty,0)}(p,\tilde{p})=N_{(1,1)}(p, \tilde{p})+1+1.\]
and we therefore have intersection multiplicity $N_{(1,1)}(p, \tilde{p})=6$.

\begin{figure}[h!]
    \subfigure[Level curves (black) with value curve (red).]
      {\includegraphics[width=0.4 \textwidth]{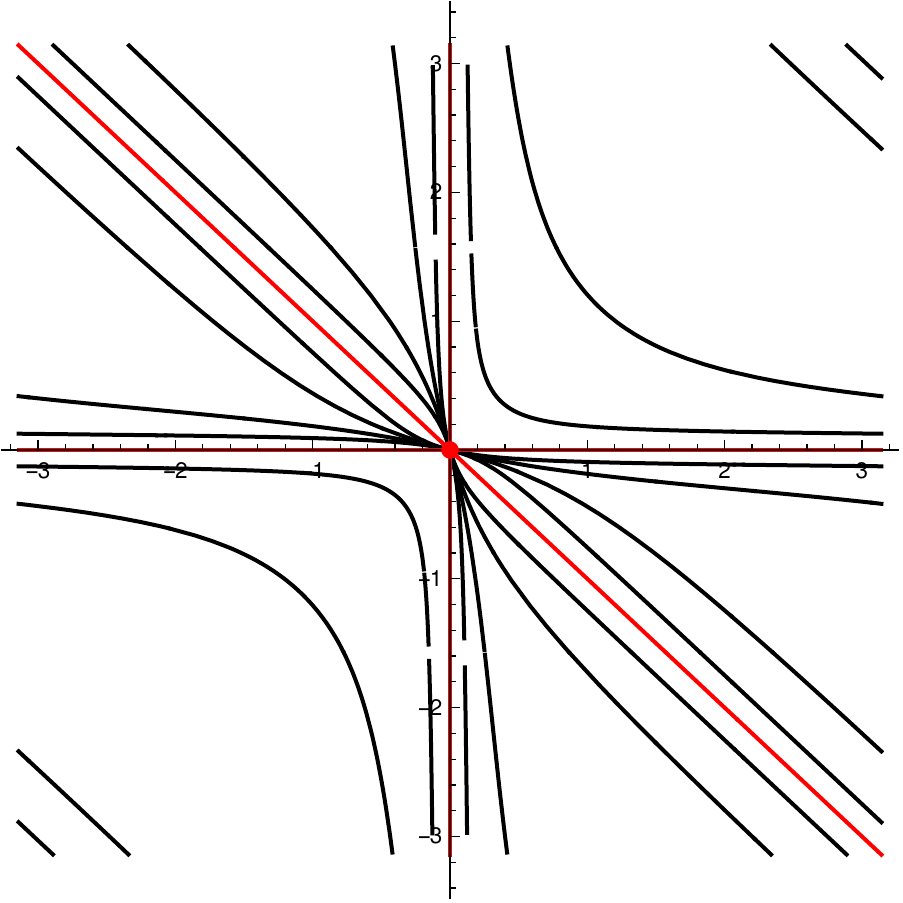}}
    \hfill
    \subfigure[Value curve bisecting second and fourth quadrants.]
      {\includegraphics[width=0.4 \textwidth]{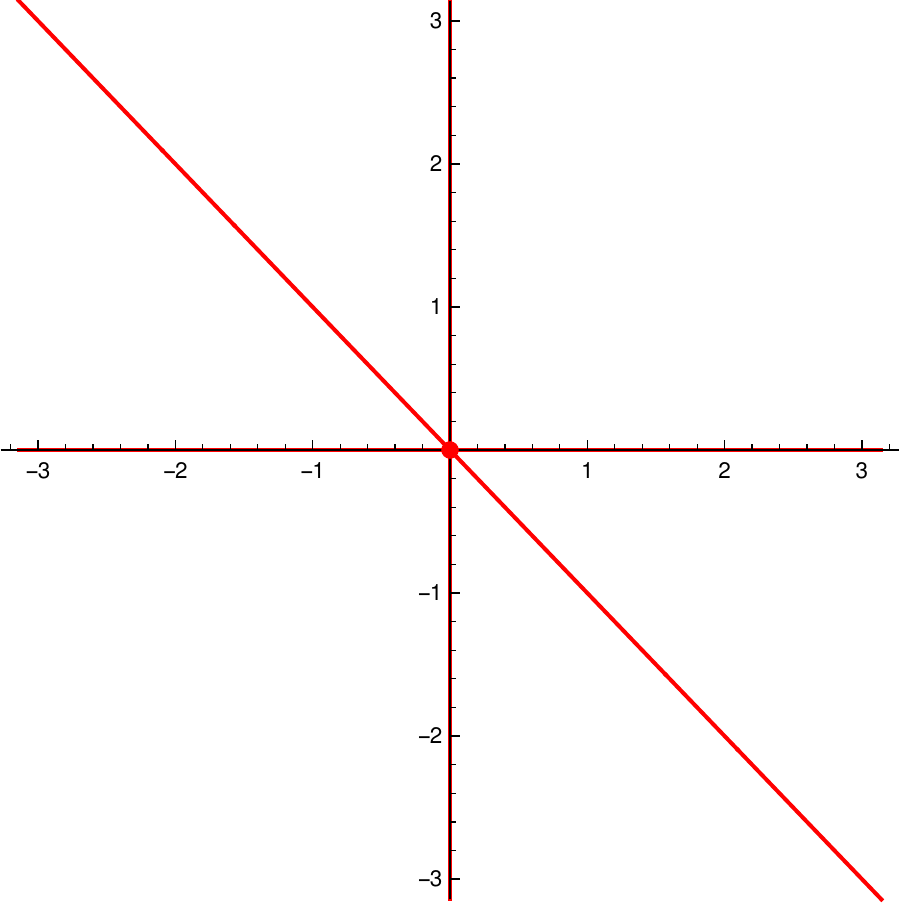}}
  \caption{\textsl{Level curves for \eqref{minimalCOex}, an RIF witnessing that contact order and intersection multiplicity are different.}
  \label{copplots}}
\end{figure}

Level lines of $\phi$ are displayed in Figure \ref{copplots}. By Theorem \ref{thm:LS}, the fact that the function in \eqref{minimalCOex} was obtained from the embedding construction implies that its value curve is given by
\[\mathcal{C}^*_{-1}(1,1)=\{(e^{t_1},1)\}\cup \{(1,e^{it_2})\} \cup \{(e^{it_1}, e^{-it_1})\}.\]
We now have two branches of the zero set of $\tilde{p}$ coming together at $(1,1)$, each with contact order equal to $2$ as can be verified directly by parametrizing the zero set $\tilde{p}(z)=0$ in terms of
\[z_2=\psi^0_{1}(z_1)=\frac{1-3z_1^2-(z-1)\sqrt{1+2z_1+9z_1^2}}{2(3z_1-4z_1^2)}\]
and
\[z_2=\psi^0_{2}(z_1)=\frac{1-3z_1^2+(z-1)\sqrt{1+2z_1+9z_1^2}}{2(3z_1-4z_1^2)}\]
and examining these functions as $\T \ni z_1\to 1$, see Figure \ref{coprootplot}. 
 Thus
\[6=N_{(1,1)}(p, \tilde{p})\neq \CO_{(1,1)}(\phi)=2,\]
as claimed, and $N_{(1,1)}(p,\tilde{p})\leq 2+2\cdot 2+2=8$, as guaranteed by Proposition \ref{prop:COvsIM}. 

Since we must have intersection multiplicity at least $6$ in order for $N_{\tau}(p, \tilde{p})$ and $\CO_{\tau}(\phi)$ to differ at a point $\tau \in \T^2$, this example is minimal in the sense of having lowest degree possible. 
\begin{figure}[h!]
    \includegraphics[width=0.4 \textwidth]{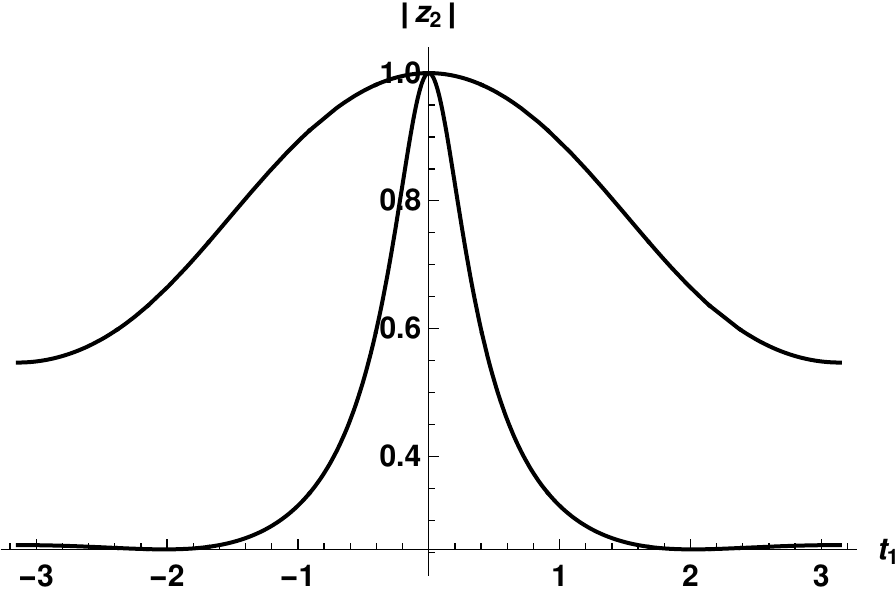}
  \caption{\textsl{Moduli of roots of $\tilde{p}(z)=0$, where $\tilde{p}$ is the numerator in \eqref{minimalCOex}, as functions of $z_1=e^{it_1} \in \T$.}}
  \label{coprootplot}
\end{figure}

\subsection{Value curves with tangential contact} 
The next example shows that value curves need not meet transversally at a singularity; we obtain it using the gluing construction in Section \ref{sec:construct}, starting  with the function $(2z_1z_2-z_1-z_2)/(2-z_1-z_2)$.
To this end, set $p(z_1,z_2)=2-z_1-z_2$, consider $r(z_1, z_2)=(p(z_1,z_2))^2+(\tilde{p}(z_1,z_2))^2$, and let
\[\tilde{P}(z_1,z_2)=z_1\frac{\partial r}{\partial z_1}(z_1,z_2)+z_2\frac{\partial r}{\partial z_2}(z_1,z_2).\]
Reflecting to obtain $P$, we arrive at the RIF
\begin{equation}
\phi(z_1,z_2)=-\frac{4z_1^2z_2^2-3z_1z_2^2-3z_1^2z_2+2z_1z_2+z_1^2+z_2^2-z_1-z_2}{4 - 
 3 z_1+z_1^2 -3 z_2 + 2 z_1z_2 -z_1^2 z_2 +z_2^2 -z_1 z_2^2}.
 \label{gluedfave}
 \end{equation}
 This RIF has a singularity at $(1,1)$ and we compute that $\phi(1,1)=1$.

This is illustrated in Figure \ref{doublefaveplots}.
\begin{figure}[h!]
    \subfigure[Level curves (black) with value curve (red).]
      {\includegraphics[width=0.4 \textwidth]{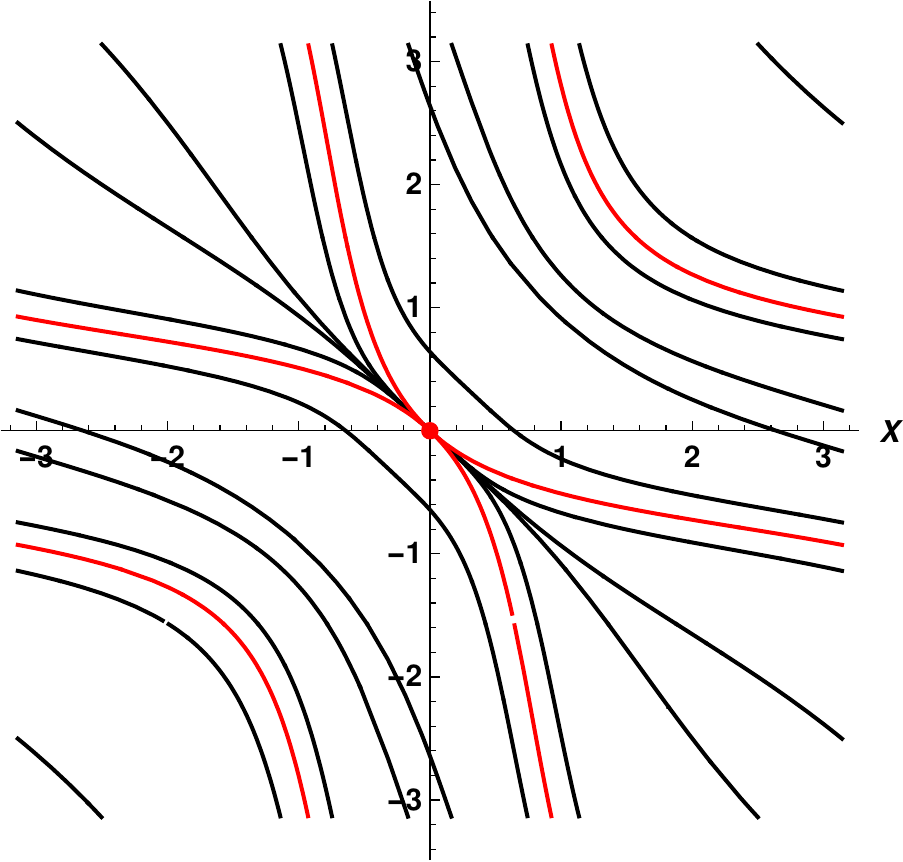}}
    \hfill
    \subfigure[Value curve having two components making contact at the origin.]
      {\includegraphics[width=0.4 \textwidth]{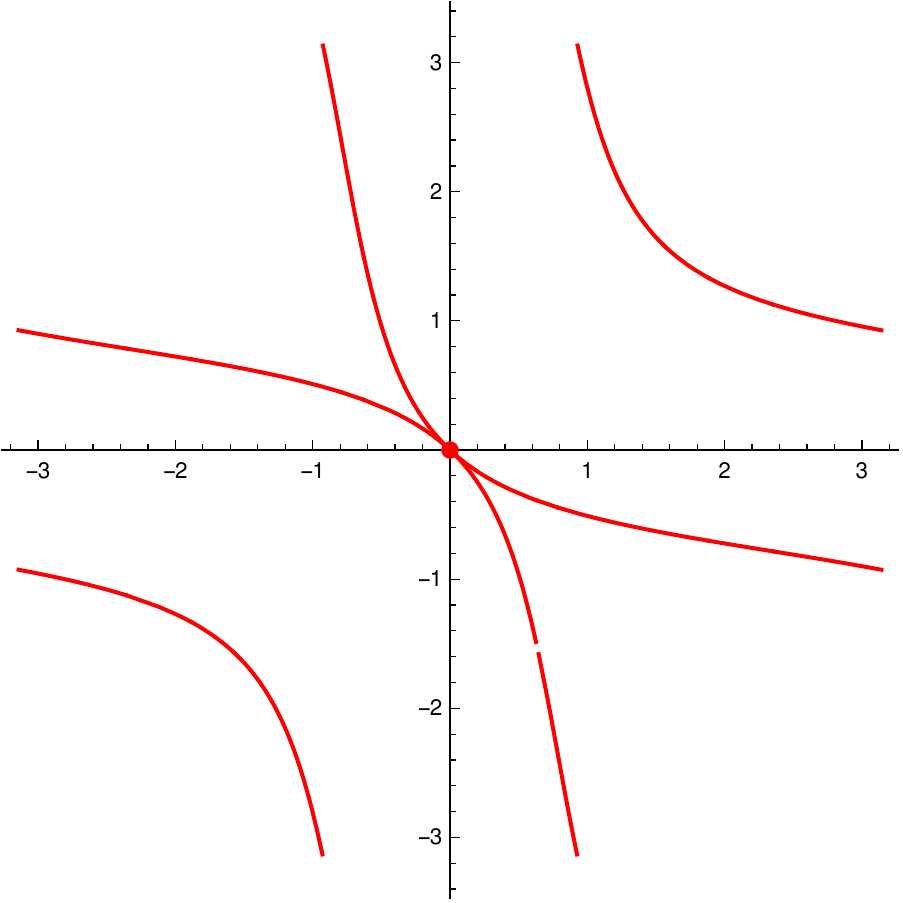}}
  \caption{\textsl{Level curves for \eqref{gluedfave}, an RIF whose value curve exhibit tangential touching}.}
  \label{doublefaveplots}
\end{figure}
By construction, the value curve of $\phi$ has two components, parametrized by reciprocals of two M\"obius transformations, namely we have
\[z_2=\psi^{i}(z_1)=i\frac{1-\frac{1+i}{2}z_1}{z_1-\frac{1-i}{2}}\quad \textrm{and}\quad z_2=\psi^{-i}(z_1)=-i\frac{1-\frac{1-i}{2}z_1}{z_1-\frac{1+i}{2}}.\]
These two curves now exhibit order $2$ tangential contact at $(1,1)$, as is guaranteed by Corollary \ref{cor:glue} and the discussion of level curves of $(2z_1z_2-z_1-z_2)/(2-z_1-z_2)$ in Section \ref{sec:intro}.

A computation using computer algebra reveals that the intersection multiplicity of $p$ and $\tilde{p}$ at $(1,1)$ is equal to $4$, and hence the contact order is equal to $4$ also. We note that $p$ and $\tilde{p}$ have four further common zeros off $\T^2$, as has to be the case in view of B\'ezout's theorem. 

Another fact illustrated by this example is that while every unimodular level curve $\mathcal{C}_{\lambda}$ passes through every singularity of $\phi$ on $\T^2$, it is not necessarily the case that every component of a level curve does: there is a pair of components in Figure \ref{doublefaveplots} (marked with ``x") that do not.
\subsection{Exceptional curves that are not value curves}\label{nonvalueexcept}
We now exhibit an RIF whose exceptional level curve $\mathcal{C}^{**}_{\mu_0}$ does not coincide with a value curve.

Consider the RIF $\phi=\tilde{p}/p$ with
\begin{equation}
p(z_1,z_2)=8 - 10 z_2 + 5 z_2^2 - z_2^3 - 10 z_1 + 10 z_1z_2 - 5 z_1z_2^2  + z_1z_2^3  + 5 z_1^2 - 
 5 z_1^2z_2 + 2 z_1^2 z_2^2 - z_1^3 + z_1^3z_2
 \label{exceptdenom}
 \end{equation}
and
\begin{equation}
\tilde{p}(z_1,z_2)=z_2^2 - z_2^3 + 2 z_1z_2 - 5 z_1z_2^2 + 5 z_1z_2^3z + z_1^2 - 5 z_1^2z_2 + 10 z_1^2 z_2^2 - 
 10 z_1^2z_2^3 - z_1^3 + 5 z_1^3z_2 - 10 z_1^3z_2^2 + 8 z_1^3z_ 2^3.
 \label{exceptnum}
 \end{equation}
This RIF is obtained by multiplying the function in \eqref{gluedfave} by $(2z_1z_2-z_1-z_2)/(2-z_1-z_2)$. 
 
One can check that $\tilde{p}$ and $p$ have a common zero at $(1,1)\in \T^2$, and eight further common zeros off $\T^2$. Moreover, $N_{(1,1)}(p, \tilde{p})=10$, as can be verified using computer algebra or by observing that all the common zeros off the two-torus have multiplicity  $1$ and using B\'ezout's theorem. We note that there are two branches of the zero set of $\tilde{p}$ coming together at $(1,1)$.

\begin{figure}[h!]
    \subfigure[Level curves (black) with value curve (red) and exceptional curve (blue).]
      {\includegraphics[width=0.4 \textwidth]{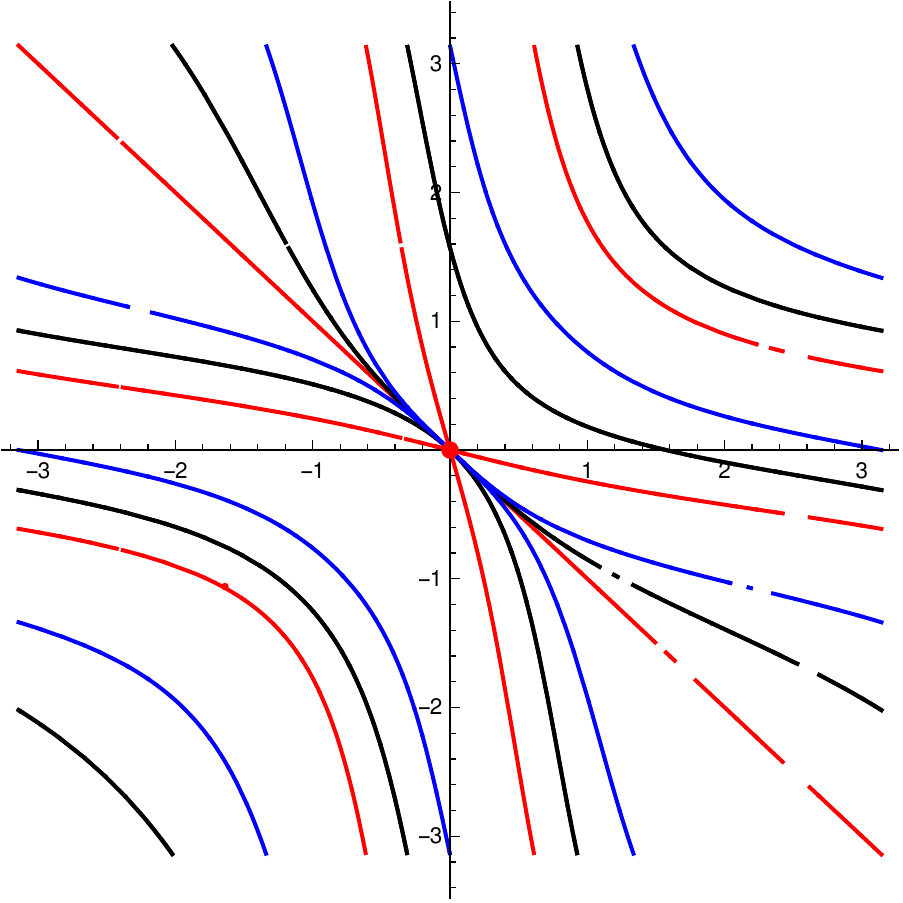}}
    \hfill
    \subfigure[Value curve (red) and exceptional curve (blue).]
      {\includegraphics[width=0.4 \textwidth]{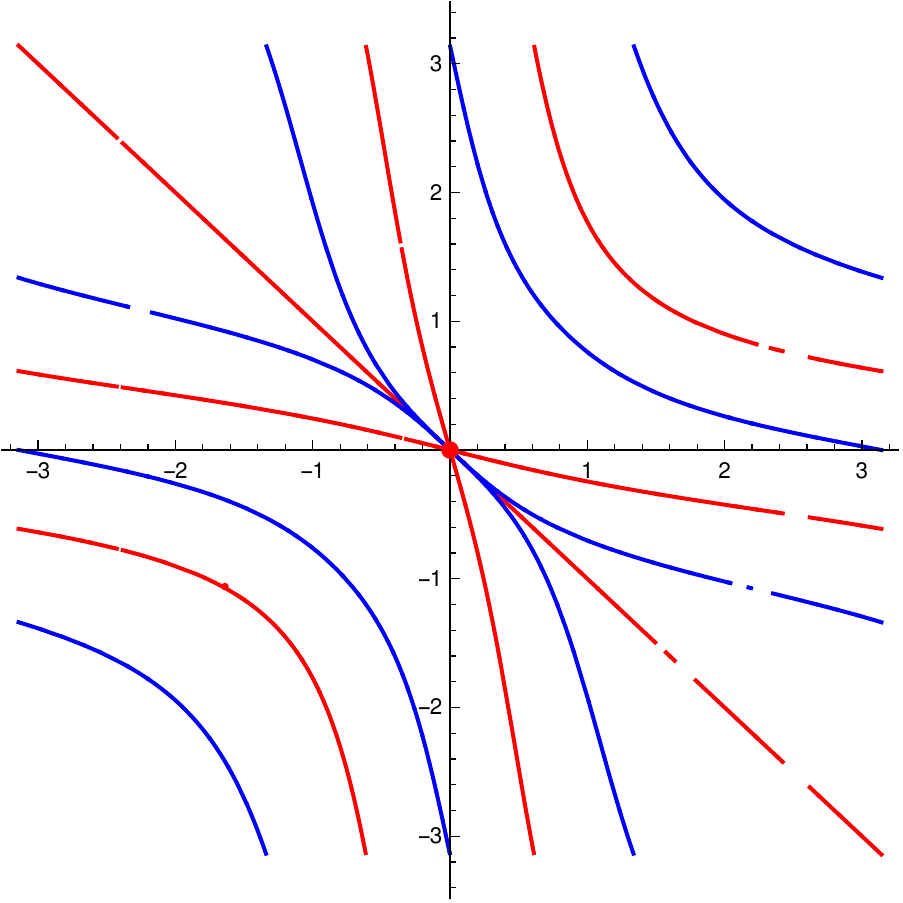}}
  \caption{\textsl{Level curves for $\phi=\tilde{p}/p$ with numerator \eqref{exceptnum} and denominator \eqref{exceptdenom}, an RIF whose exceptional curve differs from the value curve}.}
  \label{bucplots}
\end{figure}
Level lines of $\phi$ are displayed in Figure \ref{bucplots}. For this example, we have $\phi(1,1)=1$, and the
value curve contains the component $\{(e^{it_1}, e^{-it_1})\colon t_1\in (-\pi, \pi]\}$, the antidiagonal in the torus. There are two further components, which we assign indices $2$ and $3$, that are symmetric with respect to the antidiagonal, and all three components meet at $(1,1)$. 

The exceptional curve in this example is $\mathcal{C}^{**}_{-1}=\{z \in  \T^2\colon \phi(z)=-1\}$. As can be seen in Figure \ref{bucplots}, the level set $\mathcal{C}^{**}_{-1}$ has three components; note that the interlacing condition of Section \ref{sec:construct}  is satisfied by $\mathcal{C}^*_1$ and $\mathcal{C}^{**}_{-1}$. One component of $\mathcal{C}^{**}_{-1}$ omits $(1,1)$ altogether, and the two remaining components make symmetric contact with the antidiagonal. Using Lemma \eqref{lem:combolemmaOC}, we deduce that the order of contact arising from $\mathcal{C}_{-1}^{**}$ and $\mathcal{C}_1^*$ is equal to $3$. Indeed, exploiting the symmetry of $-1$-level curves  
along with the fact that they intersect two of the components of the $1$-level curve transversally, we obtain $\kappa_{1,1}^{-1,1}=3$ from
\[10=N_{(1,1)}(p,\tilde{p})=2\kappa^{-1,1}_{1,1}+2(\kappa_{1,2}^{-1,1}+\kappa^{-1,1}_{1,3})=2\kappa^{-1,1}_{1,1}+2\cdot(1+1)=2\kappa^{-1,1}_{1,1}+4.\] 

The true contact orders of individual branches of $\mathcal{Z}_{\tilde{p}}$ at $(1,1)$ are actually equal to $2$ and  $4$, respectively. This can be seen as follows. Consider the level curve $\mathcal{C}_i=\{\zeta \in \T^2\colon \phi(\zeta)=i\}$: one of the components of this level line is parametrized by
\[\psi^i(z_1)=\frac{z_1-(1+i)}{(1-i)z_1-1},\]
as can be checked by direct substitution into $\phi$, and this component (visible in black in the lower horn at  $(1,1)$ in Figure \ref{bucplots}) makes contact to order $2$ with the antidiagonal. Finally, the combinatorial formula in Lemma \eqref{lem:combolemmaOC} yields
\[10=N_{(1,1)}(p, \tilde{p})=\kappa^{i,1}_{1,1}+\kappa^{i,1}_{1,2}+\kappa^{i,1}_{1,3}+\kappa^{i,1}_{2,1}+\kappa^{i,1}_{2,2}+\kappa^{i,1}_{2,3}=2+\kappa^{i,1}_{1,2}+1+1+1+1=\kappa^{i,1}_{1,2}+6,\]
and hence $\kappa^{i,1}_{1,2}=4$.

\subsection{Multiple singularities}
The next example is constructed using the method described by the second author in \cite[Section 4]{Pas}. It shows that functions arising from that construction may have multiple singularities with different contact orders.

In the notation of \cite{Pas}, set $\mathcal{H}=\ell^2(\mathbb{Z}_4)$ and define $\pi\colon \mathbb{Z}_4\to B(\ell^2(\mathbb{Z}_4))$ by taking $\pi(j)[e_k]=e_{j+k}$ for $e_j=\delta_{j+1} \in \mathbb{Z}_4$. Set
\[A=\pi(1)+\pi(-1)=\left(\begin{array}{cccc}0 & 0 & 0 &1\\1 &0 &0&0\\0& 1& 0 & 0\\0 & 0 &0 & 1\end{array}\right)
+\left(\begin{array}{cccc}0 & 1 & 0 &0\\0 &0 &1&0\\0& 0& 0 & 1\\1 & 0 &0 & 0\end{array}\right)=
\left(\begin{array}{cccc}0 & 1 & 0 &1\\1 &0 &1&0\\0& 1& 0 & 1\\1 & 0 &1 & 0\end{array}\right)\]
and consider the diagonal matrices
\[Y=\left(\begin{array}{cccc}1 & 0 & 0 &0\\0 &1 &0&0\\0& 0& 1 & 0\\0 & 0 &0 & 0\end{array}\right)\quad 
\textrm{and} \quad z_Y=Yz_1+(1-Y)z_2=\left(\begin{array}{cccc}z_1 & 0 & 0 &0\\0 &z_1 &0&0\\0& 0& z_1 & 0\\0 & 0 &0 & z_2\end{array}\right).\]
We now obtain a Pick function via
\[f(z_1,z_2)=\langle (A-z_Y)^{-1}e_0,e_0\rangle=\frac{z_1+z_2-z_1^2z_2}{z_1(z_1^2z_2-2z_1-2z_2)}.\]
After composing with our usual M\"obius transformations $\beta$ and $\beta^{-1}$, we obtain the RIF
\begin{equation}
\phi(z_1,z_2)=-\frac{4z_1^3z_2+z_1^3-z_1^2+3z_1+1}{4+z_2-z_1z_2+3z_1^2z_2+z_1^3z_2}.
\label{bickelpascoeex}
\end{equation}

\begin{figure}[h!]
    \subfigure[Level curves, showing higher contact order at $(-1,-1)$.]
      {\includegraphics[width=0.4 \textwidth]{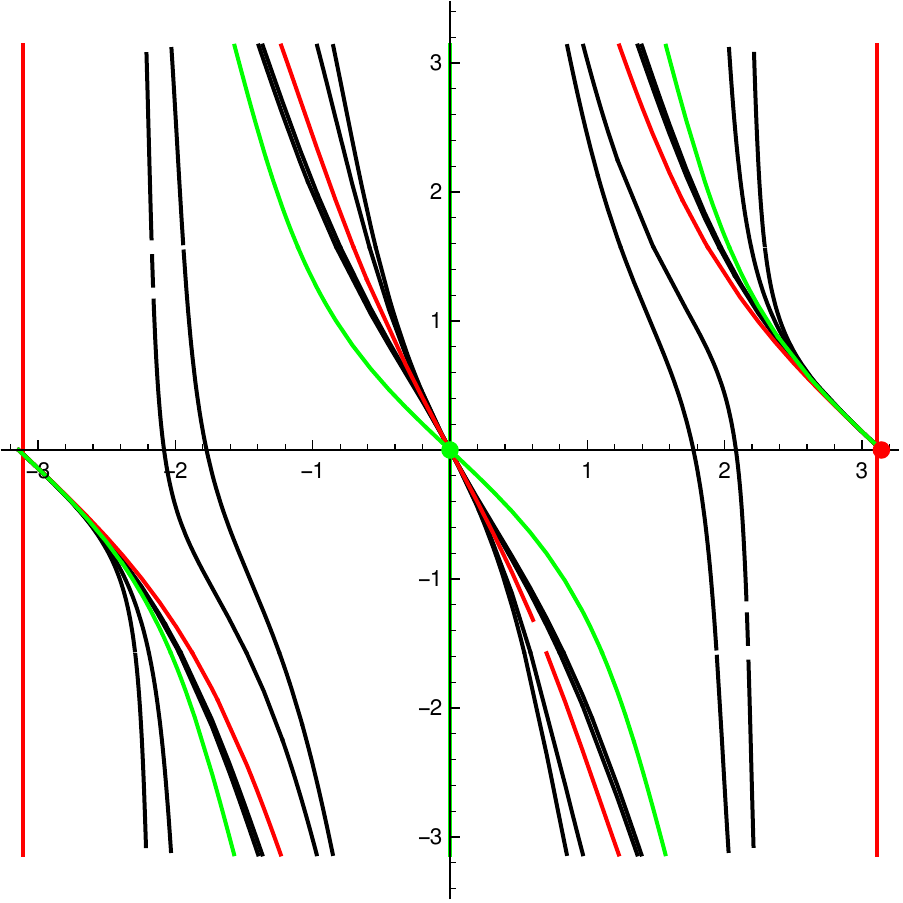}}
    \hfill
    \subfigure[Value curves $\mathcal{C}^*_{-1}(1,-1)$ (green) and $\mathcal{C}^*_1(-1,-1)$ (red).]
      {\includegraphics[width=0.4 \textwidth]{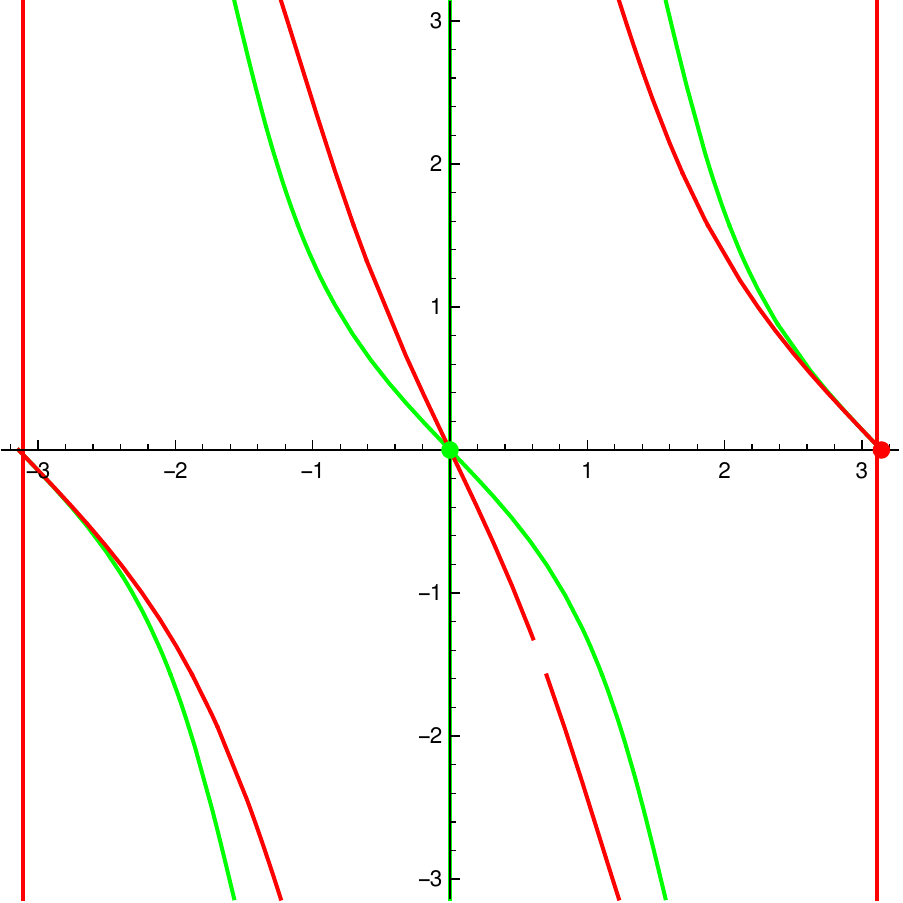}}
  \caption{\textsl{Level curves for \eqref{bickelpascoeex}, an RIF with two singularities, one with contact order $2$ and one with contact order $4$}.}
  \label{BPplots}
\end{figure}

This function has singularities at $(1,-1)$ and $(-1,-1)$ with non-tangential values $\phi(1,-1)=-1$ and  $\phi(-1,-1)=1$. Moreover, one verifies that $\CO_{(1,-1)}(\phi)=2$ and
$\CO_{(-1,-1)}(\phi)=4$. The latter contact order is essentially guaranteed by \cite[Theorem 7.1]{BPS17} and the construction, which places the Pick function $f$ in the intermediate L\"owner class $\mathcal{L}^{2-}$, but the singularity at $(1,-1)$ is in some sense extraneous.

The level curves of the function in \eqref{bickelpascoeex} are parametrized by
\[\psi^{\lambda}(z_1)=\frac{4-\lambda-3\lambda z_1+\lambda z_1^2-\lambda z_1^3}{4\lambda z_1^3-z_1^3-3z_1^2+z_1-1}\]
and are displayed in Figure \ref{BPplots} (shifted down by $\pi$ for better visibility).
Note that the value curves at $(1,-1)$ and $(-1,-1)$ contain vertical lines; the other components can be obtained by picking $\lambda$ appropriately in the parametrization $\mathcal{C}_{\lambda}$. 

In fact, value curves of degree $(n,1)$ rational inner functions with real coefficients always contain vertical lines. Assuming (A1) is satisfied, we note that $p(1,\cdot)$ and $\tilde{p}(1,\cdot)$ are linear polynomials, and then $p(1,z_2)-\tilde{p}(1,z_2)$ vanishes identically for $z_2\in \T$. Hence $p-\tilde{p}$ is divisible by $z_1-1$, and the claim follows. 

\subsection*{Acknowledgments}
Part of this work was carried out while the first and third authors were visiting Washington University in St.\ Louis, KB for the 2016-2017 academic year and AS for March-April 2017. They both thank the Wash U mathematics department, and especially John M\McC Carthy and Brett Wick, for their warm hospitality.
 

\begin{thebibliography}{alpha}
\bibliographystyle{apalike}


\bibitem[AglM\McC C]{AglMcC}J. Agler and J.E. M\McC Carthy, \emph{Pick interpolation and Hilbert function spaces}, Graduate studies in 
Amer. Math. Soc., Providence RI, 2002.

\bibitem[AglM\McC C05]{AM05}J. Agler and J.E. M\McC Carthy, Distinguished varieties, Acta Math. {\bf 194} (2005), 133-153.

\bibitem[AM\McC CS06]{AMS06}J. Agler, J.E. M\McC Carthy, and M. Stankus, Toral algebraic sets and function theory on polydisks, J. Geom. Anal. {\bf 16} (2006), no. 4, 551--562. 

\bibitem[AM\McC CS08]{AMS08}J. Agler, J.E. M\McC Carthy, and M. Stankus, Local geometry of zero sets of holomorphic functions near the torus, New York J. Math. {\bf 47} (2008),  517-538.

\bibitem[AM\McC CY12]{AMY12}J. Agler, J.E. M\McC Carthy, and N.J. Young, A Carath\'eodory theorem for the bidisk via Hilbert space methods, Math. Ann. {\bf 352} (2012), 581-624.


\bibitem[ATDY16]{ATDY16} J. Agler, R. Tully-Doyle, and N.J. Young, Nevanlinna representations in several variables, J. Funct. Anal. \textbf{270} (2016), no. 8, 3000-3046. 


\bibitem[BSV05]{BSV05} J. Ball, C. Sadosky, and V. Vinnikov, Scattering systems with several evolutions and multidimensional input/state/output systems, Integral Equations Operator Theory \textbf{52} (2005), 323--393. 

\bibitem[BKKLSS16]{BKKLSS15}C. B\'en\'eteau, G. Knese, \L. Kosi\'nski, C. Liaw, D. Seco, and
A. Sola, Cyclic polynomials in two variables,
Trans. Amer. Math. Soc. {\bf 368} (2016), 8737-8754.

\bibitem[BicGor18]{bg17} K. Bickel and P. Gorkin, Compressions of the Shift on the Bidisk and their Numerical Ranges, J. Operator Theory \textbf{79} (2018), 225-265.

\bibitem[BicKne13]{bk13}
K. Bickel and G. Knese, Inner functions on the bidisk and associated Hilbert spaces, J. Funct. Anal. \textbf{265} 
(2013), 2753--2790. 

\bibitem[BicLi17]{BLPreprint}
K. Bickel and C. Liaw, Properties of Beurling-type submodules via Agler decompositions, J. Funct. Anal. \textbf{272} (2017), 83-111.

\bibitem[BPS18]{BPS17}K. Bickel, J.E. Pascoe, and A. Sola, Derivatives of rational inner functions: geometry of singularities and integrability at the boundary, Proc. London Math. Soc. \textbf{116} (2018), 281-329.

\bibitem[CLO]{Coxetal}D.A. Cox, J. Little, and D. O'Shea, \emph{Ideals, varieties, and algorithms}, 4th ed., Undergraduate texts in mathematics, Springer-Verlag, Cham, 2015.

\bibitem[Donoghue]{DonBook}W. F. Donoghue Jr., \emph{Monotone matrix functions and analytic continuation}, Grundlehren der mathematischen Wissenschaften 207, Springer-Verlag, New York-Heidelberg, 1974. 

\bibitem[Fulton]{FulBook}W. Fulton, \emph{Algebraic curves}, Addison-Wesley Publishing Co, Redwood City, CA. Reprint of the 1969 original.

\bibitem[GKVW17]{Grinshpetal17}A. Grinshpan, D.S. Kaliuzhnyi-Verbovetskyi, V. Vinnikov, and H.J. Woerdeman, Rational inner functions on a square-matrix polyball, 267-277, in  M. Pereyra et al. (eds.), \emph{Harmonic Analysis, Partial Differential Equations, Banach Spaces, and Operator Theory}, volume 2, Springer, Cham, 2017.


\bibitem[GW06]{GW06}J.S. Geronimo and H.J. Woerdeman, Two-variable polynomials: intersecting zeros and stability, IEEE Trans. Circuits Syst. I Regul. Pap.  {\bf 53} (2006), 1130-1139.

\bibitem[JKS12]{JKS12}M.T. Jury, G. Knese, and S. McCollough, Nevanlinna-Pick interpolation on distinguished varieties in the bidisk, J. Funct. Anal. {\bf 262} (2012), 3812-3838.

\bibitem[Kne10a]{Kne10TAMS}G. Knese, Polynomials defining distinguished varieties, Trans. Amer. Math. Soc. {\bf 362} (2010), 5635-5655.

\bibitem[Kne10b]{Kne10}G. Knese, Polynomials with no zeros on the bidisk, Anal. PDE {\bf 3} (2010), 109-149.

\bibitem[Kne15]{Kne15}G. Knese, Integrability and regularity of rational functions, Proc. London. Math. Soc. {\bf 111} (2015), 1261-1306.

\bibitem[KV79]{KV79}A. Kor\'anyi and S. V\'agi, Rational inner functions on bounded symmetric domains, Trans. Amer. Math. Soc. {\bf 254} (1979), 179-193.

\bibitem[Kum02]{Kum02} A. Kummert, 2-D stable polynomials with parameter-dependent coefficients: generalizations and new results. Special issue on multidimensional signals and systems.  IEEE Trans. Circuits Systems I Fund. Theory Appl. \textbf{49} (2002), no. 6, 725--731. 

\bibitem[McD87]{McD87}J.N. McDonald, Holomorphic functions on the polydisc having positive real part, Michigan Math. J. {\bf 34} (1987), 77-84.

\bibitem[PS14]{PS14}S. Pal and O.M. Shalit, Spectral sets and distinguished varieties in the symmetrized bidisk, J. Funct. Anal. {\bf 266} (2014), 5779-5800.

\bibitem[Pas17]{Pas17} J. E. Pascoe, A wedge-of-the-edge theorem: analytic continuation of multivariable Pick functions in
and around the boundary. Bull. London Math. Soc. \textbf{49} (2017) 916-925.

\bibitem[Pas]{Pas}J.E. Pascoe, An inductive Julia-Carath\'eodory theorem for Pick functions in two variables, Proc. Edinb. Math. Soc., to appear. Preprint available at arxiv.org/abs/1605.08707.

\bibitem[Rudin]{Rud69} W. Rudin, \emph{Function Theory in polydisks},
W. A. Benjamin, Inc., New York-Amsterdam, 1969.

\bibitem[Rud69]{RudTAMS69}W. Rudin, Pairs of inner functions on finite Riemann surfaces, Trans. Amer. Math. Soc. {\bf 140} (1969), 423-434.

\bibitem[RS65]{RudStout65}W. Rudin and E.L. Stout, Boundary properties of functions of several complex variables, J. Math. Mech. {\bf 14} (1965), 991-1005.

\bibitem[Wag11]{Wag11} D.G. Wagner, Multivariate stable polynomials: theory and applications, Bull. Amer. Math. Soc. (N.S.) \textbf{48} (2011), no. 1, 53--84. 

\end{thebibliography}
\end{document}